\documentclass[12pt]{amsart}

\usepackage[colorlinks=true]{hyperref}
\usepackage{amssymb}
\usepackage[all]{xy}
\usepackage{mathtools}
\usepackage{stackrel}
\usepackage{stmaryrd}

\usepackage{todonotes}
\usepackage{eucal}

\usepackage{tikz}
\usetikzlibrary{calc,matrix}
\tikzset{w/.style={circle, draw,inner sep=1pt},b/.style={circle,draw,fill,inner sep=2pt}, s/.style={rectangle, draw,inner sep=3pt}}

\newcommand{\nocontentsline}[3]{}
\newcommand{\tocless}[2]{\bgroup\let\addcontentsline=\nocontentsline#1{#2}\egroup}

\newcommand{\lra}{\longrightarrow}

\newcommand{\BD}{\mathbb{BD}}
\newcommand{\bD}{\mathbb{D}}
\newcommand{\bE}{\mathbb{E}}

\newcommand{\bL}{\mathbb{L}}

\newcommand{\bP}{\mathbb{P}}

\newcommand{\bT}{\mathbb{T}}

\newcommand{\Z}{\mathbb{Z}}

\newcommand{\cB}{\mathcal{B}}
\newcommand{\cC}{\mathcal{C}}

\newcommand{\cM}{\mathcal{M}}
\newcommand{\cO}{\mathcal{O}}
\newcommand{\cP}{\mathcal{P}}

\newcommand{\B}{\mathrm{B}}
\newcommand{\C}{\mathrm{C}}
\newcommand{\E}{\mathrm{E}}
\newcommand{\rH}{\mathrm{H}}
\newcommand{\rN}{\mathrm{N}}
\newcommand{\T}{\mathrm{T}}

\newcommand{\ga}{\gamma}
\newcommand{\e}{\epsilon}
\newcommand{\la}{\lambda}
\newcommand{\si}{\sigma}
\newcommand{\fd}{\mathfrak{d}}
\newcommand{\g}{\mathfrak{g}}

\newcommand{\Ass}{\mathrm{Ass}}
\newcommand{\Arr}{\mathrm{Arr}}

\newcommand{\Cois}{\mathrm{Cois}}

\newcommand{\Comm}{\mathrm{Comm}}
\newcommand{\Comp}{\mathrm{Comp}}

\newcommand{\dAff}{\mathrm{dAff}}
\newcommand{\dgmix}{\mathrm{dg}^{gr,\epsilon}}
\newcommand{\DR}{\mathrm{DR}}
\newcommand{\bDR}{\mathbf{DR}}

\newcommand{\Fun}{\mathrm{Fun}}

\newcommand{\Hom}{\mathrm{Hom}}

\newcommand{\id}{\mathrm{id}}
\newcommand{\Isot}{\mathrm{Isot}}
\newcommand{\Lagr}{\mathrm{Lagr}}
\newcommand{\Map}{\mathrm{Map}}

\newcommand{\MC}{\mathrm{MC}}
\newcommand{\Obs}{\mathrm{Obs}}
\newcommand{\uObs}{\underline{\mathrm{Obs}}}
\newcommand{\Perf}{\mathrm{Perf}}
\newcommand{\Pois}{\mathrm{Pois}}
\newcommand{\Pol}{\mathrm{Pol}}
\newcommand{\pt}{\mathrm{pt}}
\newcommand{\Rep}{\mathrm{Rep}}
\newcommand{\SC}{\mathrm{SC}}
\newcommand{\sSet}{\mathrm{sSet}}
\newcommand{\St}{\mathrm{St}}
\newcommand{\Sym}{\mathrm{Sym}}
\newcommand{\Symp}{\mathrm{Symp}}
\newcommand{\TCois}{\mathrm{TCois}}

\newcommand{\alg}{\mathrm{Alg}}
\newcommand{\balg}{\mathbf{Alg}}
\newcommand{\bdg}{\mathbf{dg}}
\newcommand{\dg}{\mathrm{dg}}
\newcommand{\CAlg}{\mathrm{CAlg}}
\newcommand{\bCAlg}{\mathbf{CAlg}}

\newcommand{\bDef}{\mathbf{Def}}

\newcommand{\Lie}{\mathrm{Lie}}
\newcommand{\bfPol}{\boldsymbol{\mathrm{Pol}}}

\newcommand{\bU}{\mathbf{U}}

\newcommand{\bZ}{\mathbf{Z}}

\newcommand{\bimod}[2]{{}_{#1}\mathbf{BMod}_{#2}}
\newcommand{\blmod}{\mathbf{LMod}}
\newcommand{\brmod}{\mathbf{RMod}}
\renewcommand{\bmod}{\mathbf{Mod}}
\newcommand{\comod}{\mathrm{CoMod}}

\DeclareMathOperator{\gr}{gr}
\DeclareMathOperator{\Spec}{Spec}

\newcommand{\llpar}{(\!(}
\newcommand{\rrpar}{)\!)}

\renewcommand{\d}{\mathrm{d}}
\newcommand{\ddr}{\mathrm{d}_{\mathrm{dR}}}

\newtheorem{thm}{Theorem}[section]
\newtheorem*{theorem}{Theorem}
\newtheorem{prop}[thm]{Proposition}
\newtheorem{lm}[thm]{Lemma}
\newtheorem{cor}[thm]{Corollary}
\newtheorem{conjecture}[thm]{Conjecture}

\theoremstyle{definition}
\newtheorem{defn}[thm]{Definition}

\theoremstyle{remark}
\newtheorem{remark}[thm]{Remark}
\newtheorem{example}[thm]{Example}

\begin{document}
\title{Derived coisotropic structures II: stacks and quantization}
\address{Dipartimento di Matematica, Universit\`a di Milano, Via Cesare Saldini 50, 20133 Milan, Italy}
\email{valerio.melani@outlook.com}
\author{Valerio Melani}
\address{Department of Mathematics, University of Geneva, 2-4 rue du Lievre, 1211 Geneva, Switzerland}
\curraddr{Institut f\"{u}r Mathematik, Winterthurerstrasse 190, 8057 Z\"{u}rich, Switzerland}
\email{pavel.safronov@math.uzh.ch}
\author{Pavel Safronov}
\begin{abstract}
We extend results about $n$-shifted coisotropic structures from part I of this work to the setting of derived Artin stacks. We show that an intersection of coisotropic morphisms carries a Poisson structure of shift one less. We also compare non-degenerate shifted coisotropic structures and shifted Lagrangian structures and show that there is a natural equivalence between the two spaces in agreement with the classical result. Finally, we define quantizations of $n$-shifted coisotropic structures and show that they exist for $n>1$.
\end{abstract}
\maketitle

\tableofcontents

\addtocontents{toc}{\protect\setcounter{tocdepth}{1}}

\section*{Introduction}

This paper is a continuation of \cite{MS1} where we have defined a notion of an $n$-shifted coisotropic structure on a morphism of commutative dg algebras. In this paper we extend this definition to derived Artin stacks. Some of our results are as follows:
\begin{itemize}
\item An intersection of $n$-shifted coisotropic morphisms carries an $(n-1)$-shifted Poisson structure.

\item A non-degenerate $n$-shifted coisotropic structure is the same as an $n$-shifted Lagrangian structure.

\item Let $f\colon L \to X$ be a morphism of derived stacks equipped with an $n$-shifted coisotropic structure. If $n>1$, then $f$ admits a canonical deformation quantization. 
\end{itemize}

\subsection*{Shifted Poisson algebras}

Recall the operad $\bP_n$ which controls commutative dg algebras equipped with a Poisson bracket of degree $1-n$. An important feature of the operad $\bP_n$ is the \emph{additivity property}, which is the equivalence
\[\balg_{\bP_{n+1}}\cong \balg(\balg_{\bP_n})\]
of symmetric monoidal $\infty$-categories. Such an equivalence has been constructed by Rozenblyum and, independently, by the second author in \cite{Sa2}. In other words, one can think of a $\bP_{n+1}$-algebra as an associative algebra object in $\bP_n$-algebras. Therefore, we can define an action of a $\bP_{n+1}$-algebra $A$ on a $\bP_n$-algebra $B$ to be simply the data of a structure of a left $A$-module on $B$. Moreover, in \cite{MS1} we have constructed an explicit two-colored operad $\bP_{[n+1, n]}$ which models such actions, i.e. there is an equivalence of $\infty$-categories
\[\balg_{\bP_{[n+1, n]}}\cong \blmod(\balg_{\bP_n}),\]
where $\blmod$ is the $\infty$-category of pairs of an associative algebra and a left module.

Given a commutative algebra $A$, one defines the space $\Pois(A, n)$ of $n$-shifted Poisson structures on $A$ to be the space of lifts of $A$ to a $\bP_{n+1}$-algebra. This space has the following alternative description. There is a graded $\bP_{n+2}$-algebra $\bfPol(A, n)$ of $n$-shifted polyvectors on $A$, and it was shown in \cite{Me} that there is an equivalence of spaces
\[\Pois(A, n)\cong \Map_{\balg_{\Lie}^{gr}}(k(2)[-1], \bfPol(A, n)[n+1]),\]
where $k(2)[-1]$ is the trivial graded Lie algebra concentrated in weight $2$ and cohomological degree $1$.

A similar definition was given for $n$-shifted coisotropic structures in \cite{MS1}. Suppose $f\colon A\rightarrow B$ is a morphism of commutative dg algebras. We can consider $f$ as an object of $\blmod(\bCAlg)$, where $\bCAlg$ is the $\infty$-category of commutative dg algebras. In other words, $f$ endows $B$ with an action of $A$. Then the space $\Cois(f, n)$ of $n$-shifted coisotropic structures on $f\colon A\rightarrow B$ is the space of lifts of $f$ along the forgetful functor
\[\blmod(\balg_{\bP_n})\longrightarrow \blmod(\bCAlg).\]
One can construct a graded $\bP_{[n+2, n+1]}$-algebra $\bfPol(f, n) = (\bfPol(A, n), \bfPol(B/A, n-1))$ and it was shown in \cite[Theorem 4.15]{MS1} that there is an equivalence of spaces
\[\Cois(f, n)\cong \Map_{\balg_{\Lie}^{gr}}(k(2)[-1], \bfPol(f, n)[n+1]).\]
Note that we can identify
\begin{align*}
\bfPol(A, n)&\cong \Hom_A(\Sym_A(\bL_A[n+1]), A)\\
\bfPol(B/A, n-1)&\cong \Hom_B(\Sym_B(\bL_{B/A}[n]), B)
\end{align*}
as graded commutative algebras and the action on the level of commutative algebras is induced from the morphism $\bL_{B/A} \to \bL_A\otimes_A B [1]$.

\subsection*{Shifted Poisson structures on stacks}

Recall that a Poisson structure on a smooth scheme $X$ is defined to be the structure of a $k$-linear Poisson algebra on the structure sheaf $\cO_X$. If $X$ is a derived Artin stack, we no longer have the structure sheaf $\cO_X$ as an object of a category of sheaves of $k$-modules, so it is not clear how to extend the above definition of an $n$-shifted Poisson structure to stacks.

To a derived Artin stack $X$ we can associate its de Rham stack $X_{DR}$ together with a projection $q\colon X\rightarrow X_{DR}$. Since the cotangent complex of the de Rham stack $X_{DR}$ is trivial, it is reasonable to expect that an $n$-shifted Poisson structure on $X$ is the same as a relative $n$-shifted Poisson structure on $q\colon X\rightarrow X_{DR}$. This simplifies the problem as now the fibers of $q$ are affine formal derived stacks, in the sense of \cite[Section 2.2]{CPTVV}. Even though they are not affine schemes, it is shown in \cite{CPTVV} that they are controlled by a certain \emph{graded mixed} commutative dg algebra.

More precisely, the general theory of \emph{formal localization} developed in \cite[Section 2]{CPTVV} produces a graded mixed commutative algebra $\bD_{X_{DR}}$ enhancing the structure sheaf $\cO_{X_{DR}}$ and a graded mixed commutative algebra $\cB_X$ enhancing the pushforward $q_*\cO_X$. Then an $n$-shifted Poisson structure on a derived Artin stack $X$ is a lift of $\cB_X$ to a $\bD_{X_{DR}}$-linear $\bP_{n+1}$-algebra. Note that one has to introduce certain twists $\cB_X(\infty)$ and $\bD_{X_{DR}}(\infty)$ to fully capture all polyvectors, but we ignore this technical difference in the introduction.

The same procedure works almost verbatim in the relative setting and we can define the space $\Cois(f, n)$ of $n$-shifted coisotropic structures on a morphism $f\colon L\rightarrow X$ of derived Artin stacks (see Definition \ref{def:generalcoisotropics}). Note that for our purposes it is useful to include the data of an $n$-shifted Poisson structure on $X$ in the definition of the space $\Cois(f, n)$ and not fix it in advance.

One can similarly define the notion of relative polyvectors $\bfPol(f, n)$ which is again a graded $\bP_{[n+2, n+1]}$-algebra associated to a morphism $f\colon L\rightarrow X$ of derived Artin stacks. Extending \cite[Theorem 4.15]{MS1} to the setting of derived stacks, we obtain the following result (see Theorem \ref{thm:spaceofcoisotropics2}):
\begin{theorem}
Let $f\colon L\rightarrow X$ be a morphism of derived Artin stacks. Then we have an equivalence of spaces
\[\Cois(f, n)\cong \Map_{\balg_{\Lie}^{gr}}(k(2)[-1], \bfPol(f, n)[n+1]).\]
\end{theorem}

Here are some examples of coisotropic structures described in Section \ref{sect:coisotropicexamples}.
\begin{itemize}
\item (Classical case). Suppose $f\colon L\hookrightarrow X$ is a smooth closed subscheme of a smooth scheme $X$. Then we show that the \emph{space} of $0$-shifted Poisson structures on $X$ is equivalent to the \emph{set} of ordinary Poisson structures on $X$ and the \emph{space} of $0$-shifted coisotropic structures on $f\colon L\hookrightarrow X$ is equivalent to the \emph{subset} of the set of Poisson structures on $X$ for which $L$ is coisotropic in the classical sense.

\item (Identity). We show that the space of $n$-shifted coisotropic structures on the identity morphism $\id\colon X\rightarrow X$ is equivalent to the space of $n$-shifted Poisson structures on the target. In other words, identity has a unique coisotropic structure. This has an interesting consequence: the forgetful morphisms
\[
\Pois(X, n-1)\longleftarrow \Cois(\id, n)\stackrel{\sim}\longrightarrow \Pois(X, n)
\]
between spaces of shifted coisotropic and shifted Poisson structures assemble to give a forgetul map $\Pois(X, n)\rightarrow \Pois(X, n-1)$ from $n$-shifted Poisson structures on $X$ to $(n-1)$-shifted Poisson structures on $X$. This map is nontrivial in general even though the underlying bivector of the corresponding $(n-1)$-shifted Poisson structure can be shown to be zero.

\item (Graph). Suppose $X$ and $Y$ are derived Artin stacks equipped with $n$-shifted Poisson structures. Moreover, suppose $f\colon X\rightarrow Y$ is a morphism compatible with the Poisson structures. Then we show that the graph $X\rightarrow \overline{X}\times Y$ carries a canonical $n$-shifted coisotropic structure, where $\overline{X}$ is the same stack as $X$ but equipped with the opposite $n$-shfited Poisson structure. In fact, this gives a complete characterization of $n$-shifted coisotropic structures on the graph (see Theorem \ref{thm:graphpoisson}).
\end{itemize}

In \cite{PTVV} it was shown that given two Lagrangians $L_1, L_2\rightarrow X$ in an $n$-shifted symplectic stack $X$, their derived intersection $L_1\times_X L_2$ carries a canonical $(n-1)$-shifted symplectic structure. We extend this result to coisotropic structures in the following statement (see Theorem \ref{thm:generalintersections}).

\begin{theorem}
Suppose $X$ is an $n$-shifted Poisson stack and $L_1, L_2\rightarrow X$ are two morphisms equipped with compatible $n$-shifted coisotropic structures. Then the intersection $L_1\times_X L_2$ carries a natural $(n-1)$-shifted Poisson structure such that the natural projection
\[L_1\times_X L_2\longrightarrow \overline{L}_1\times L_2\]
is a morphism of $(n-1)$-shifted Poisson stacks.
\end{theorem}

We remark that this statement gives a nice conceptual explanation of the main result of \cite{BG}, generalizing it to a much broader context.

\subsection*{Non-degenerate coisotropic structures}

One may ask more generally how the theory of shifted symplectic and shifted Lagrangian structures of \cite{PTVV} relates to the theory of shifted Poisson and shifted coisotropic structures. Classically, a Poisson structure whose bivector induces an isomorphism $\T^*_X\stackrel{\sim}\rightarrow \T_X$ is the same as a symplectic structure. It was shown in \cite{CPTVV} and \cite{Pri1} that the subspace $\Pois^{nd}(f, n)\subset \Pois(f, n)$ of non-degenerate $n$-shifted Poisson structures, i.e. those that induce an equivalence $\bL_X\stackrel{\sim}\rightarrow \bT_X[-n]$, is equivalent to the space $\Symp(X, n)$ of $n$-shifted symplectic structures.

Suppose that $f\colon L\rightarrow X$ is equipped with an $n$-shifted coisotropic structure. Then we obtain a natural morphism of fiber sequences
\[
\xymatrix{
\bL_{L/X}[-1] \ar[d] \ar[r] & f^*\bL_X \ar[r] \ar[d] & \bL_L \ar[d] \\
\bT_L[-n] \ar[r] & f^*\bT_X[-n] \ar[r] & \bT_{L/X}[1-n]
}
\]

It is thus natural to define the subspace $\Cois^{nd}(f, n)\subset \Cois(f, n)$ of non-degenerate $n$-shifted coisotropic structures to be those that induce equivalences $\bL_X\rightarrow \bT_X[-n]$ and $\bL_{L/X}\rightarrow \bT_L[1-n]$ (which automatically implies that $\bL_L\rightarrow \bT_{L/X}[1-n]$ is an equivalence as well). We prove the following result (see Theorem \ref{thm:coisnd=lagr}).

\begin{theorem}
Suppose $f\colon L\rightarrow X$ is a morphism of derived Artin stacks. Then we have an equivalence
\[\Cois^{nd}(f, n)\cong \Lagr(f, n)\]
of spaces of non-degenerate $n$-shifted coisotropic structures and $n$-shifted Lagrangian structures on $f$.
\end{theorem}

Let us mention that the proof for $n=0$ was previously given by Pridham in \cite{Pri4} using a slightly different notion of coisotropic structures. We closely follow his proof to obtain the result for all $n$ for our definition of $n$-shifted coisotropic structures. The main idea is to prove a stronger result by showing that there is an equivalence of spaces equipped with a natural (co)filtration: $\Lagr(f, n)$ is filtered by the maximal weight of the form and $\Cois^{nd}(f, n)$ is filtered by the maximal weight of the polyvector. The proof then proceeds by induction by developing obstruction theory where the inductive step is a simple problem in linear algebra.

We believe an alternative proof can be given along the lines of the proof of \cite{CPTVV} by using the Darboux lemma for shifted Lagrangians from \cite{JS}.

\subsection*{Quantization}

We conclude this paper with a description of deformation quantization of shifted coisotropic structures. Recall that a deformation quantization of a $\bP_n$-algebra $A$ can be formulated in terms of lifts of $A$ to a $\BD_n$-algebra (the Beilinson--Drinfeld operad $\BD_n$ is reviewed in Section \ref{sect:BDoperads}). Since the notion of an $n$-shifted Poisson structure on a stack is reduced to a $\bP_{n+1}$-algebra structure on $\cB_X$, one can similarly define the notion of a deformation quantization of an $n$-shifted Poisson stack. One has the following results on quantizations of $n$-shifted Poisson structures on derived Artin stacks:
\begin{itemize}
\item If $n \geq 1$, it is shown in \cite{CPTVV} that every $n$-shifted Poisson stack admits a deformation quantization by using the formality of the $\bE_n$ operad.

\item If $n=0$, it is shown in \cite{Pri3} that non-degenerate 0-shifted Poisson structures (i.e. $0$-shifted symplectic structures) admit a \emph{curved} deformation quantization. In fact, the paper gives a complete characterization of such deformation quantizations and shows that they are unobstructed beyond the first order in $\hbar$.

\item If $n=-1$, it is shown in \cite{Pri2} that non-degenerate $(-1)$-shifted Poisson structures (i.e. $(-1)$-shifted symplectic structures) admit deformation quantizations in a twisted sense if $X$ is Gorenstein with a choice of the square root $\omega_X^{1/2}$ of the dualizing sheaf.
\end{itemize}

We formulate the notion of deformation quantization for $n$-shifted coisotropic structures and show that using the formality of the $\bE_n$ operad one can prove the following result (see Theorem \ref{thm:coisotropicdefquantization}).

\begin{theorem}
Suppose $n\geq 2$. Then any $n$-shifted coisotropic structure on a morphism $f\colon L\rightarrow X$ of derived Artin stacks admits a deformation quantization.
\end{theorem}

Let us mention that \cite{Pri4} shows that non-degenerate $0$-shifted coisotropic structures (i.e. $0$-shifted Lagrangians) admit curved deformation quantizations if we again include a certain twist necessary to deal with the first-order obstruction (see also \cite{BGKP}). Moreover, we expect that the results of \cite{Pri4} can be extended to similarly provide quantizations of $1$-shifted Lagrangians thus treating the only remaining case.

One might wonder how the above theorem relates to the statement of the non-formality of the Swiss-cheese operad $\SC_n$ shown by Livernet \cite{Li}. Recall that the operad $\bE_n$ has two filtrations: the cohomological one and the one related to the Poisson operad. These coincide for $n\geq 2$ but differ for $n=1$ and $n=0$. For instance, the cohomological filtration for $\bE_1 \cong \Ass$ is trivial while the Poisson filtration is nontrivial (it is the so-called PBW filtration). Livernet shows that the cohomological filtration on $\C_\bullet(\SC_n)$ is nontrivial for $n\geq 2$. We expect that there is similarly a Poisson filtration on $\C_\bullet(\SC_n)$ whose associated graded is $\bP_{[n, n-1]}$. It is obvious that the two filtrations are different for $n\leq 2$ by looking at the associated gradeds; we expect that the two filtrations moreover differ for all $n$. That is, the Poisson filtration on $\C_\bullet(\SC_n)$ is trivial for $n\geq 3$ while the cohomological filtration is nontrivial. Let us also note a related drastic difference between the operads $\bP_n$ and $\bP_{[n+1, n]}$: while $\bP_n$ admits a minimal model with a quadratic differential, the differential on the minimal model $\tilde{\bP}_{[n+1, n]}$ for $\bP_{[n+1, n]}$ constructed in \cite[Section 3.4]{MS1} has higher terms.

\subsection*{Acknowledgements}

We would like to thank D. Calaque, G. Ginot, M. Porta, B. To\"{e}n and G. Vezzosi for many interesting and stimulating discussions. The work of P.S. was supported by the EPSRC grant EP/I033343/1.
The work of V.M. was partially supported by FIRB2012 "Moduli spaces and applications". We also thank the anonymous referee for his/her useful comments.
\addtocontents{toc}{\protect\setcounter{tocdepth}{2}}

\section{Recollections from part I}

In this section we recall some results from \cite{CPTVV} and \cite{MS1} that will be used in the paper.

\subsection{Notations}

We will use the following notations used in \cite{MS1}:
\begin{itemize}
\item $k$ denotes a field of characteristic zero. $\dg_k$ is the category of cochain complexes of $k$-vector spaces considered with its projective model structure. The standard tensor product on $\dg_k$ makes it a symmetric monoidal model category. The associated symmetric monoidal $\infty$-category will be denoted by $\bdg_k$.

\item We denote by $M$ a symmetric monoidal model dg category as in \cite[Section 1.1]{MS1} and by $\cM$ its underlying $\infty$-category.

\item $M^{gr}$ and $M^{gr,\epsilon}$ denote the model categories of (weight) graded objects in $M$ and of graded mixed objects in $M$ respectively. We refer to \cite[Section 1.2]{MS1} and \cite[Section 1.1]{CPTVV} for more details. The associated $\infty$-categories will be denoted by $\cM^{gr}$ and $\cM^{gr,\epsilon}$ respectively.

\item If $\cC$ is any $\infty$-category, then $\cC^{\sim}$ denotes its $\infty$-groupoid of equivalences. We will also refer to $\cC^{\sim}$ as the underlying space of $\cC$. The $\infty$-category of morphisms in $\cC$ will be denoted by $\Arr(\cC)=\Fun(\Delta^1, \cC)$.

\item Given a dg operad $\cP$, the category of $\cP$-algebras in $M$ is denoted by $\alg_{\cP}(M)$. The $\infty$-category of $\cP$-algebras in $\cM$ is denoted by $\balg_{\cP}(\cM)$. In the case of $\cP$ being the operad $\Ass$ of associative algebras, we will simply use the notation $\alg(M)$ and $\balg(\cM)$. Similarly, for the operad $\Comm$ of commutative algebras we will use $\CAlg(M)$ and $\bCAlg(\cM)$.

\item When considering Lie and Poisson algebras in $M^{gr}$ or $\cM^{gr}$ we will always assume that the bracket is of weight $-1$. In the case $M = \dg_k$, we will use the simpler notations $\alg_{\Lie}^{gr}$ and $\balg_{\Lie}^{gr}$ instead of the full $\alg_{\Lie}(\dg^{gr}_k)$ and $\balg_{\Lie}(\bdg_k^{gr})$.

\item We denote by $\bimod{}{}(\cM)$ the $\infty$-category of triples $(A,B,L)$, where $A,B \in \balg(\cM)$ and $L$ is an $(A,B)$-bimodule. In a similar way, we denote by $\brmod(\cM)$ (resp. $\blmod(\cM)$) the $\infty$-category of pairs $(A,L)$ where $A \in \balg(\cM)$ and $L$ is a right (resp. left) $A$-module.

\item By a derived Artin stack we mean a derived Artin stack locally of finite presentation over $k$.
\end{itemize}

\subsection{Formal localization and Poisson structures on derived stacks}

Let $X$ be a derived Artin stack. Recall that the de Rham stack $X_{DR}$ is defined to be 
\[X_{DR}(A) = X(\rH^0(A)^{red})\]
for any cdga $A$ concentrated in non-positive cohomological degrees. It comes equipped with a projection $q \colon X \to X_{DR}$, whose fibers are formal completions of $X$.

Moreover, there are two naturally defined prestacks of graded mixed commutative algebras on $X_{DR}$, denoted by $\bD_{X_{DR}}$ and $\cB_X$; they are to be thought as derived versions of the crystalline structure sheaf and of the sheaf of principal parts respectively. More precisely, we have the following equivalences of prestacks of commutative algebras on $X_{DR}$:
\[|\bD_{X_{DR}}|\cong \cO_{X_{DR}},\qquad |\cB_X|\cong q_* \cO_X.\]
Just as in the classical case, we have a morphism $\bD_{X_{DR}} \to \cB_X$ that we think of as a $\bD_{X_{DR}}$-linear structure on $\cB_X$.

As functors to the category of graded mixed modules, both $\bD_{X_{DR}}$ and $\cB_X$ admit natural twistings $\bD_{X_{DR}}(\infty)$ and $\cB_X(\infty)$, that are now prestacks of commutative algebras in Ind-objects in the category of graded mixed modules. For details on these constructions, see \cite[Sections 1.5 and 2.4.2]{CPTVV}. Notice that in particular $\cB_X(\infty)$ is a commutative algebra in the category of $\bD_{X_{DR}}(\infty)$-modules. 

With these notations, we can now give the definition of shifted Poisson structures, see \cite[Section 3.1]{CPTVV}.

\begin{defn}
\label{def:poissonstack}
Let $X$ be a derived Artin stack. The \emph{space $\Pois(X,n)$ of n-shifted Poisson structures} on $X$ is the space of lifts of the commutative algebra $\cB_X(\infty)$ in the $\infty$-category of $\bD_{X_{DR}}(\infty)$-modules to a $\bP_{n+1}$-algebra.
\end{defn}

Note that if $X=\Spec A$ is an affine derived scheme, this definition recovers \cite[Definition 4.4]{MS1}. Just as in the affine case, we can give an alternative definition of shifted Poisson structures. First, define the graded $\bP_{n+2}$-algebra of $n$-shifted polyvectors on $X$ to be
\[\bfPol(X,n) := \Gamma(X_{DR}, \bfPol^t (\cB_X / \bD_{X_{DR}},n))\]
where $\bfPol^t (\cB_X / \bD_{X_{DR}},n)$ is the Tate realization of the algebra of shifted $\bD_{X_{DR}}$-linear multiderivations of $\cB_X$. Again, we refer to \cite{CPTVV} for more details on this construction. Using the main theorem of \cite{Me} and its extended version \cite[Theorem 4.5]{MS1} one immediately obtains the following result (see also \cite[Theorem 3.1.2]{CPTVV}).

\begin{thm}
\label{thm:spaceofpoisson}
With notations as above, there is a canonical equivalence of spaces
$$\Pois(X,n) \cong \Map_{\alg^{gr}_{\Lie}}(k(2)[-1], \bfPol(X,n)[n+1]).$$
\end{thm}

We remark that the above theorem is somewhat reassuring, since it gives an expected alternative description of Poisson structures on derived stacks in terms of bivectors.

\begin{example}
Suppose $X$ is a smooth scheme and $n=0$. Then \[\bfPol(X, 0)\cong \Gamma(X, \Sym(\T_X[-1]))\] where $\T_X$ is the tangent bundle of $X$, and the $\bP_2$-structure on the right is given by the Schouten bracket. The completion $\bfPol(X, 0)^{\geq 2}$ in weights at least 2 is concentrated in degree at least 2. Using \cite[Proposition 1.19]{MS1}, we have an equivalence of spaces
\[\Pois(X, 0)\cong \underline{\MC}(\bfPol(X, 0)^{\geq 2}[1]).\]

Since the Lie algebra $\bfPol(X, 0)^{\geq 2}[1]$ is concentrated in degree at least 1, the space of Maurer--Cartan elements is discrete. Its elements correspond to bivectors $\pi_X\in \mathrm{H}^0(X, \wedge^2 \T_X)$ satisfying $[\pi_X, \pi_X]=0$, i.e. we recover the usual notion of a Poisson structure.

The same argument shows that there are no nontrivial $n$-shifted Poisson structures on smooth schemes for $n>0$.
\end{example}

\subsection{Coisotropic structures on algebras}

Recall from \cite[Section 3.4]{MS1} the colored operad $\bP_{[n+1,n]}$ whose algebras in $\cM$ are pairs of objects $(A,B)$ in $\cM$ together with the following additional structure:
\begin{itemize}
\item a $\bP_{n+1}$-structure on $A$;
\item a $\bP_{n}$-structure on $B$;
\item a morphism of $\bP_{n+1}$-algebras $A\to \bZ(B)$, where $\bZ(B)$ is the Poisson center of $B$, considered with its natural structure of a $\bP_{n+1}$-algebra in $\cM$.
\end{itemize}

Recall also that there is a canonical morphism of commutative algebras $\bZ(B)\to B$. The fiber of this map is denoted by $\bDef(B)[-n]$, and we have a fiber sequence
\[B[-1] \lra \bDef(B)[-n] \lra \bZ(B)\]
of non-unital $\bP_{n+1}$-algebras (see \cite[Section 3.5]{MS1}). In particular, it follows that if $(A,B)$ is a $\bP_{[n+1,n]}$-algebra, then the fiber $\bU(A,B)$ of the underlying morphism of commutative algebras $A \to B$ fits in the Cartesian square
\[
\xymatrix{
\bU(A,B) \ar[r] \ar[d] & A \ar[d] \\
\bDef(B)[-n] \ar[r] & \bZ(B) \\
}
\]
of non-unital $\bP_{n+1}$-algebras in $\cM$. In particular, we obtain a forgetful functor
\[\balg_{\bP_{[n+1, n]}}\longrightarrow \balg_{\bP^{\mathrm{nu}}_{n+1}}.\]

Using the morphism of commutative algebras $\bZ(B)\to B$, we also get a natural forgetful map $$\phi \colon \  \balg_{\bP_{[n+1,n]}}(\cM)^{\sim} \to \mathrm{Arr}(\bCAlg(\cM))^{\sim}$$
to the $\infty$-groupoid of morphisms of commutative algebras, sending a $\bP_{[n+1,n]}$-algebra $(A,B)$ to the underlying map $A\to B$. This forgetful functor was used in \cite[Section 4.3]{MS1} to define coisotropic structures on a morphism $f\colon A \to B$ in $\bCAlg(\cM)$. The following is \cite[Definition 4.12]{MS1}.

\begin{defn}\label{defn:cois}
Let $f\colon A\to B$ be a map of commutative algebra objects in the $\infty$-category $\cM$. The \emph{space $\Cois(f, n)$ of $n$-shifted coisotropic structures on $f$} is the fiber of the forgetful functor $\phi$, taken at $f$.
\end{defn}

We also have an alternative operadic description of the space $\Cois(f,n)$. The following is \cite[Theorem 2.22]{Sa2}.
\begin{thm}\label{thm:additivity}
Let $\cM$ be a symmetric monoidal dg category satisfying our starting assumption. Then there is an equivalence of $\infty$-categories
$$\balg_{\bP_{n+1}}(\cM) \simeq \balg(\balg_{\bP_n}(\cM)).$$
\end{thm}
Thanks to the above result one can think of a $\bP_{n+1}$-algebra as an associative algebra in the $\infty$-category of $\bP_n$-algebras. In particular, this allows us to obtain the following important theorem, which is \cite[Corollary 3.8]{Sa2} (see also \cite[Theorem 4.14]{MS1}), which gives an alternative characterization of coisotropic structures.

\begin{thm}
There is a commutative diagram of $\infty$-categories
\[
\xymatrix{
\balg_{\bP_{[n+1, n]}}(\cM) \ar^{\sim}[rr] \ar[dr] && \blmod(\balg_{\bP_n}(\cM)) \ar[dl] \\
& \Arr(\bCAlg(\cM))
}
\]
\end{thm}

We immediately get the following corollary about coisotropic structures.

\begin{cor}\label{cor:coiscptvv}
Let $f\colon A \to B$ be a morphism in $\bCAlg(\cM)$. Then the space $\Cois(f,n)$ is equivalent to the fiber of
$$\blmod(\balg_{\bP_n}(\cM))^{\sim} \lra \Arr(\bCAlg(\cM))^{\sim}$$
taken at $f$.
\end{cor} 
In other words, consider $f\colon A \to B$ in $\bCAlg(\cM)$. Then a $n$-shifted coisotropic structure on $f$ amounts to a $\bP_{n+1}$-structure on $A$, a $\bP_n$-structure on $B$ and an action of $A$ on $B$ as $\bP_n$-algebras, all compatible with the given commutative morphism $f$.

We end this section by recalling a third important characterization of the space of coisotropic structures on algebras.
Let again $f\colon A \to B$ be a morphism in $\bCAlg(\cM)$. It induces a natural morphism of graded commutative algebras
\[\bfPol(A,n) \lra \bfPol(B/A,n-1)\]
between algebras of shifted polyvectors. In \cite[Section 4.2]{MS1} we showed that the pair $(\bfPol(A, n), \bfPol(B/A, n-1))$ becomes a $\bP_{[n+2, n+1]}$-algebra and we denote the underlying non-unital $\bP_{n+2}$-algebra by $\bfPol(f, n)$. We also have its internal version, i.e. there is an $\bP_{[n+2, n+1]}$-algebra $(\bfPol^{int}(A, n), \bfPol^{int}(B/A, n-1))$ in $\cM$ and we denote its underlying non-unital $\bP_{n+2}$-algebra in $\cM$ by $\bfPol^{int}(f, n)$. The following is \cite[Theorem 4.15]{MS1}.

\begin{thm}\label{thm:coispolyvectors}
With notations as above there is an equivalence of spaces
\[\Cois(f,n) \simeq \Map_{\alg_{\Lie}^{gr}}(k(2)[-1], \bfPol(f,n)[n+1]).\]
\end{thm}

\section{Coisotropic structures on derived stacks}

In this section we generalize definitions of coisotropic structures from the affine case to the case of general stacks.

\subsection{Definitions}

Let $f\colon L \to X$ be a map of derived Artin stacks. The map $f$ descends to a map between the de Rham stacks $f_{DR} \colon L_{DR} \to X_{DR}$, which in turn induces a pullback functor (simply denoted by $f^*$, with a slight abuse of notation) from prestacks on $X_{DR}$ to prestacks on $L_{DR}$. The functor $f^*$ is just the precomposition of prestacks with $f$.

By definition of $\bD_{X_{DR}}$, one immediately gets an equivalence $\bD_{L_{DR}} \cong f^* \bD_{X_{DR}}$. As for the sheaves of principal parts, $f$ induces a natural algebra map
\[f^*_{\cB} \colon f^*\cB_X \to \cB_L\]
preserving the $\bD_{L_{DR}}$-linear structures. It follows that there exists an induced morphism
\[ f^*_{\cB}(\infty) \colon \ f^*\cB_X(\infty) \to \cB_L(\infty)\]
of $\bD_{L_{DR}}(\infty)$-algebras. 

Let us denote by $\cC_X$ the $\infty$-category of $\bD_{X_{DR}}(\infty)$-modules in the $\infty$-category of functors 
$$(\mathrm{dAff}/X_{DR})^{op} \longrightarrow \mathrm{Ind}(\bdg_k^{gr,\epsilon}).$$
Similarly, we let $\cC_L$ be the $\infty$-category of $\bD_{L_{DR}}(\infty)$-modules in the category of functors 
$$(\mathrm{dAff}/L_{DR})^{op} \longrightarrow \mathrm{Ind}(\bdg_k^{gr, \epsilon}).$$
By the above discussion, the map $f^*_{\cB}(\infty)$ is naturally a morphism of commutative algebras in $\cC_L$.

Moreover, $f$ induces a symmetric monoidal pullback $\infty$-functor $\cC_X \to \cC_L$, so that in particular there is a well defined functor
$$\balg_{\bP_{n+1}}(\cC_X) \longrightarrow \balg_{\bP_{n+1}}(\cC_L).$$
For example, suppose that $X$ is endowed with a $n$-shifted Poisson structure. This corresponds to a $\bP_{n+1}$-structure on $\cB_X(\infty)$ in the $\infty$-category $\cC_{X}$, so that $f^*\cB_X(\infty)$ becomes a $\bP_{n+1}$-algebra in $\cC_L$. This means that there is an induced map of spaces
$$\Pois(X,n) \longrightarrow \Pois(f^*\cB_X(\infty), n),$$
where $\Pois(f^*\cB_X(\infty),n)$ is the space of compatible $\bP_{n+1}$-structures on $f^*\cB_X(\infty)$, viewed as an element in $\bCAlg(\cC_L)$.

\begin{defn}\label{def:generalcoisotropics}
Let $f\colon L\to X$ be a map of derived Artin stacks. The \emph{space $\Cois(f,n)$ of $n$-shifted coisotropic structures on $f$} is the fiber product
$$\xymatrix{
\Cois(f,n) \ar[r]\ar[d] & \Pois(X,n) \ar[d] \\
\Cois(f^*_{\cB}(\infty),n) \ar[r] & \Pois(f^*\cB_X(\infty),n)
}$$
where $\Cois(f^*_{\cB}(\infty),n)$ is the space of coisotropic structures on $f^*_{\cB}(\infty)$ in the sense of Definition \ref{defn:cois}.
\end{defn}

\begin{remark}
Our definition is equivalent to the one given in \cite[Section 3.4]{CPTVV} if one identifies $\bP_{[n+1, n]}$-algebras with $\bP_{(n+1, n)}$-algebras following \cite[Theorem 3.7]{Sa2}.
\end{remark}

Recall from \cite[Section 4.3]{MS1} that if $\cM$ is a nice enough symmetric monoidal $\infty$-category, then for every $f\colon A\to B$ in $\bCAlg(\cM)$ the space $\Cois(f,n)$ comes equipped with two natural forgetful maps
\[
\xymatrix{
& \Cois(f, n) \ar[dl] \ar[dr] & \\
\Pois(B, n-1) && \Pois(A, n).
}
\]

In particular, it follows that for every map $f\colon L \to X$ of derived Artin stacks we have a similar correspondence of spaces
\[
\xymatrix{
& \Cois(f, n) \ar[dl] \ar[dr] & \\
\Pois(L, n-1) && \Pois(X, n).
}
\]

\subsection{Relative polyvectors for derived stacks}

In this subsection we give an alternative definition of coisotropic structures on a morphism of derived Artin stacks using the notion of relative polyvectors. The goal is to prove an analogue of Theorem \ref{thm:coispolyvectors} in the more general case of derived stacks. 

For every map of derived stacks $f\colon L \to X$ the morphism
\[f^*_{\cB}\colon f^*\cB_X \to \cB_L\]
we introduced in the previous subsection is a map of commutative algebras in the $\infty$-category of $\bD_{L_{DR}}$-modules. Similarly, after twisting we get a morphism of commutative algebras
\[f^*_{\cB}(\infty) \colon f^*\cB_X(\infty) \to \cB_L(\infty)\]
in the $\infty$-category $\cC_L$ of $\bD_{L_{DR}}(\infty)$-modules. In particular, we can consider its algebra of $n$-shifted relative polyvectors. That is, we have a graded $\bP_{[n+2, n+1]}$-algebra
\[(\bfPol^{int}(f^*\cB_X(\infty), n), \bfPol^{int}(\cB_L(\infty)/f^*\cB_X(\infty), n-1))\]
in $\cC_L$.

\begin{remark}
	The fact that the $\infty$-category $\cC_L$ can be presented as a model category (and thus that we can use the constructions and the results of \cite{MS1}) is a consequence of \cite[Section 2.3.2]{TV}. Alternatively, one can observe that in this particular case polyvectors are in fact functorial, and hence can be defined objectwise (see \cite[Remark 2.4.8]{CPTVV} for more details). 
\end{remark}

We have a graded commutative algebra
\[\bfPol(L/X, n) = \Gamma(L, \Sym_{\cO_L}(\bT_{L/X}[-n-1]))\]
as in \cite[Definition 2.3.7]{CPTVV}. The following lemma is a straightforward consequence of formal localization, as studied in \cite[Section 2]{CPTVV}.

\begin{lm}
There is an equivalence of graded commutative cdgas
\[\bfPol(L/X, n)\cong \Gamma(L_{DR}, \bfPol^{int}(\cB_L(\infty)/f^*\cB_X(\infty), n-1)).\]
\end{lm}

The morphism of $\bP_{n+2}$-algebras
\[\bfPol(X,n) \cong \Gamma(X_{DR}, \bfPol^{int}(\cB_X(\infty), n)) \lra \Gamma(L_{DR}, \bfPol^{int}(f^*\cB_X(\infty),n))\]
moreover induces a graded $\bP_{[n+2, n+1]}$-algebra structure on the pair
\[(\bfPol(X, n), \bfPol(L/X, n-1)).\]

\begin{defn}
Let $f\colon L\rightarrow X$ be a morphism of derived Artin stacks. The \emph{algebra of relative $n$-shifted polyvectors} is the graded non-unital $\bP_{n+2}$-algebra
\[\bfPol(f, n) = \bU(\bfPol(X, n), \bfPol(L/X, n-1)).\]
\end{defn}

\begin{prop}
For a morphism $f\colon L\rightarrow X$ of derived Artin stacks there is a fiber sequence of graded non-unital $\bP_{n+2}$-algebras
\[\bfPol(L/X, n-1)[-1]\longrightarrow \bfPol(f, n)\longrightarrow \bfPol(X, n).\]

The connecting homomorphism $\bfPol(X, n)\rightarrow \bfPol(L/X, n-1)$ is induced from the morphism $\bL_{L/X}\rightarrow f^*\bL_X[1]$.
\label{prop:relpolyvectorsstacks}
\end{prop}
\begin{proof}
The first claim follows from definitions since we have a fiber sequence of graded non-unital $\bP_{n+2}$-algebras
\[\bfPol(L/X, n-1)[-1]\longrightarrow \bU(\bfPol(X, n), \bfPol(L/X, n-1))\longrightarrow \bfPol(X, n).\]

Moreover, we have a fiber sequence of graded non-unital commutative algebras
\[\bU(\bfPol(X, n), \bfPol(L/X, n-1))\longrightarrow \bfPol(X, n)\longrightarrow \bfPol(L/X, n-1)\]
and the second claim follows from the fact that by the general machinery of formal localization of \cite[Section 2]{CPTVV}, the map
$$ \bL_{\cB_L(\infty)/f^*\cB_X(\infty)}^{int} \to  \bL_{f^*\cB_X(\infty)}^{int}[1] \otimes_{f^*\cB_X(\infty)}\cB_L(\infty). $$
corresponds exactly to the morphism $\bL_{L/X} \to f^*\bL_X[1]$.
\end{proof}

We are now ready to prove our first main result relating the space of $n$-shifted coisotropic structures of Definition \ref{defn:cois} to the algebra of relative polyvectors introduced above. The following theorem is an extension of Theorem \ref{thm:coispolyvectors} to derived Artin stacks.

\begin{thm}\label{thm:spaceofcoisotropics2}
Let $f\colon L \to X$ be a map of derived Artin stacks. Then we have an equivalence of spaces
\[\Cois(f,n) \simeq \Map_{\alg_{\Lie}^{gr}}(k(2)[-1], \bfPol(f,n)[n+1]).\]
\end{thm}
\begin{proof}
Let again $\cC_L$ be the $\infty$-category of $\bD_{L_{DR}}(\infty)$-modules. As above, there is a morphism of commutative algebras in $\cC_L$
$$f^*_{\cB}(\infty) \colon \   f^*\cB_X(\infty) \to \cB_L(\infty)\,$$
whose algebra of relative $n$-shifted polyvectors fits into a fiber sequence 
$$ \bfPol(\cB_L(\infty) / f^* \cB_X(\infty),n-1)[n]\longrightarrow  \bfPol(f^*_{\cB}(\infty),n)[n+1] \longrightarrow   \bfPol(f^*\cB_X(\infty),n)[n+1]$$
of graded Lie algebras. 

By definition, the graded $\bP_{n+2}$-algebra $\bfPol(f,n)$ fits into a Cartesian square
\[ \xymatrix{
\bfPol(f,n) \ar[r] \ar[d] & \bfPol(X,n) \ar[d] \\
\bfPol(f^*_{\cB}(\infty),n) \ar[r] & \bfPol(f^*\cB_X(\infty), n)
}\]
in the category of graded non-unital $\bP_{n+2}$-algebras

Moreover, it follows from Theorem \ref{thm:coispolyvectors} applied in $\cC_L$ that the space of $n$-shifted coisotropic structures on $f^*_{\cB}(\infty)$ has an explicit description in terms of $\bfPol(f^*_{\cB}(\infty),n)$; namely, one has an equivalence
\[ \Cois(f^*_{\cB}(\infty),n) \cong \Map_{\alg_{\Lie}^{gr}}(k(2)[-1] , \bfPol(f^*_{\cB}(\infty),n)[n+1]).\]

On the other hand, Theorem \ref{thm:spaceofpoisson} tells us that
\[ \Pois(X,n) \simeq \Map_{\alg_{\Lie}^{gr}}(k(2)[-1], \bfPol(X,n)[n+1]), \]
and similarly
\[ \Pois(f^*\cB_X(\infty),n) \simeq \Map_{\alg_{\Lie}^{gr}}(k(2)[-1], \bfPol(f^*\cB_X(\infty),n)[n+1]), \]
so that we immediately get the desired equivalence.
\end{proof}

The alternative characterization of coisotropic structures given by Theorem \ref{thm:spaceofcoisotropics2} is of a more geometric nature than Definition \ref{def:generalcoisotropics}. This demonstrates why this definition is a sensible generalization of the classical notion, as explained in the following examples.

\subsection{Examples}
\label{sect:coisotropicexamples}

\begin{enumerate}
\item \textbf{Smooth schemes.} Let $L$ be a smooth subscheme of a smooth scheme $X$, and let $f\colon L \to X$ be the corresponding immersion. Suppose $X$ is endowed with a classical Poisson structure $\pi_X$, i.e. $\pi_X\in\Pois(X, 0)$. The graded $\bP_2$-algebra of $0$-shifted polyvectors on $X$ is
\[\bfPol(X, 0)\cong \Gamma(X, \Sym_{\cO_X}(\T_X[-1])),\]
where the weight grading is given by putting the tangent bundle $\T_X$ in weight 1 and the bracket is the usual Schouten bracket. Denote by $\rN_{L/X}\cong \bT_{L/X}[1]$ the normal bundle to the subscheme $L$.

By Proposition \ref{prop:relpolyvectorsstacks} we get a graded $L_\infty$ structure on the graded complex
\[\bfPol(f, 0)[1]\cong \Gamma(X, \Sym(\T_X[-1]))[1]\oplus \Gamma(L, \Sym(\rN_{L/X}[-1]))\]
with the differential twisted by the morphism $f^*$.

A Maurer--Cartan element in $\bfPol(f, 0)^{\geq 2}[1]$ is an element $\pi_X\in \rH^0(X, \wedge^2 \T_X)$, so let us analyze the possible brackets of such an element. The bracket $[\pi_X, ..., \pi_X]_n$ has degree $2$ and weight $n+1$. Therefore, $[\pi_X, ..., \pi_X]_n = 0$ for $n > 2$. The projection $\bfPol(f, 0)[1]\rightarrow \bfPol(X, 0)[1]$ has a structure of a graded $L_\infty$ morphism, hence $[\pi_X, \pi_X]$ in $\bfPol(f, 0)^{\geq 2}[1]$ is the standard Schouten bracket. Let us denote by $f^*\pi_X$ the image of $\pi_X$ in $\Gamma(L, \wedge^2(\rN_{L/X}))$. Then the Maurer--Cartan equation for $\pi_X$ in $\bfPol(f, 0)^{\geq 2}[1]$ splits into two:
\[[\pi_X, \pi_X] = 0\quad f^* \pi_X = 0.\]

The first equation is the integrability equation for the Poisson structure on $X$ and the second equation is equivalent to $L\rightarrow X$ being coisotropic with respect to the Poisson structure $\pi_X$.

By degree reasons the space of Maurer--Cartan elements in $\bfPol(f, 0)^{\geq 2}[1]$ is discrete and hence $\Cois(f, 0)$ is a subset of $\Pois(X, 0)$ of Poisson structures for which the subscheme $L$ is coisotropic in the usual sense. In other words, in the classical context the morphism $L\rightarrow X$ has a coisotropic structure iff $L\rightarrow X$ is a coisotropic submanifold in the usual sense.

Here is an alternative way to obtain the same conclusion. Assume first $L$ and $X$ are affine. In this case the morphism
\[f^*\colon \Gamma(X, \Sym(\T_X[-1]))\longrightarrow \Gamma(L, \Sym(\rN_{L/X}[-1]))\]
is surjective and hence by \cite[Proposition 4.11]{MS1} $\bfPol(f, 0)[1]$ is equivalent to the algebra of polyvectors $\Pol(f, 0)[1]$ on $X$ which vanish when pulled back to $\rN_{L/X}$. This gives the result in the affine case and the general case follows by descent.

\item \textbf{Identity.} \label{ex:identity} Let $X$ be a derived Artin stack and consider the identity morphism $\id\colon X\rightarrow X$. The homotopy fiber sequence of graded dg Lie algebras
\[\bfPol(X/X, n-1)[n] \rightarrow \bfPol(\id, n)[n+1] \rightarrow \bfPol(X, n)[n+1]\]
implies that the projection $\bfPol(\id, n)\rightarrow \bfPol(X, n)$ is a quasi-isomorphism in weights $\geq 1$ since $\bT_{X/X}=0$. Therefore, the natural projection
\[\Cois(\id, n)\longrightarrow \Pois(X, n)\]
is a weak equivalence, i.e. the identity morphism has a unique coisotropic structure for any $n$-shifted Poisson structure on $X$.

An interesting consequence of this statement is that we obtain a forgetful map $\Pois(X, n)\rightarrow \Pois(X, n-1)$ given as the composite
\[\Pois(X, n)\cong \Cois(\id, n)\longrightarrow \Pois(X, n-1).\]

\item \textbf{Point.} Let $X$ be a derived Artin stack and consider the projection $p\colon X\rightarrow \pt$. The homotopy fiber sequence of graded dg Lie algebras
\[\bfPol(X/\pt, n-1)[n] \rightarrow \bfPol(p, n)[n+1]\rightarrow \bfPol(\pt, n)[n+1]\]
implies that the morphism $\bfPol(X, n-1)[n]\rightarrow \bfPol(p, n)[n+1]$ is a quasi-isomorphism in weights $\geq 1$. Therefore, the natural morphism
\[\Cois(p, n)\rightarrow \Pois(X, n-1)\]
is a weak equivalence.

Note that this is a shifted Poisson analogue of a well-known statement for shifted symplectic structures: a Lagrangian structure on $X\rightarrow \pt$ where the point is equipped with its unique $n$-shifted symplectic structure is the same as an $(n-1)$-shifted symplectic structure on $X$.
\end{enumerate}

Let us now give a more general procedure to construct coisotropic structures. Let $X,Y$ be derived Artin stacks together with a morphism $f\colon X\rightarrow Y$. In analogy with \cite[Section 4.5]{MS1}, we give the following definition.
\begin{defn}\label{defn:poissonmaps} Let $f\colon X \to Y$ be a morphism of derived Artin stacks, and let
$$f^*_{\cB}(\infty) \colon f^*\cB_{Y}(\infty) \to \cB_X(\infty)$$
be the induced map of $\bD_{X_{DR}}(\infty)$-algebras. We define \emph{the space $\Pois(f,n)$ of $n$-Poisson structures on $f$} to the the fiber product
$$\xymatrix{ \Pois(f,n) \ar[r] \ar[d] & \Pois(X,n) \times \Pois(Y,n)\ar[d] \\
\Pois(f^*_{\cB}(\infty),n) \ar[r] & \Pois(\cB_X(\infty),n) \times \Pois(f^*\cB_Y(\infty),n)
}$$
where $\Pois(f^*_{\cB}(\infty),n)$ is defined as in \cite[Definition 4.18]{MS1}.
\end{defn}

In the special case where $X= \Spec (B)$ and $Y=\Spec(A)$ are derived affine schemes, one sees that $\Pois(f,n)$ is indeed equivalent to the fiber of
$$ \Arr(\balg_{\bP_{n+1}}(\bdg_k))^{\sim} \lra \Arr(\bCAlg(\bdg_k))^{\sim} $$
taken at the given map $f\colon A \to B$. This means Definition \ref{defn:poissonmaps} is a generalization of \cite[Definition 4.18]{MS1}.

Denote by $g\colon X\rightarrow X\times Y$ the graph of $f$, that is to say the morphism given by $\id\times f$. The following result is an extension of \cite[Theorem 4.20]{MS1} to the case of general derived stacks.

\begin{thm}\label{thm:graphpoisson}
With notations as above, there is a cartesian square of spaces
$$\xymatrix{
\Pois(f,n) \ar[r] \ar[d] & \Pois(X,n) \times \Pois(Y,n) \ar[d] \\
\Cois(g,n) \ar[r] & \Pois(X \times Y,n) }$$
where the map on the right sends two Poisson structures $\pi_X$ and $\pi_Y$ to the Poisson structure given by $(\pi_X;-\pi_Y)$.
\end{thm}
\begin{proof}
Consider the map of $\bD_{X_{DR}}(\infty)$-algebras
$$f^*_{\cB}(\infty) \colon f^*\cB_Y(\infty) \to \cB_X(\infty).$$
By \cite[Theorem 4.20]{MS1}, we know that there is a fiber square of spaces
$$\xymatrix{
\Pois(f^*_{\cB}(\infty),n) \ar[r]\ar[d] & \Pois(\cB_X(\infty),n) \times \Pois(f^*\cB_Y(\infty),n) \ar[d] \\
\Cois(g^*_{\cB}(\infty),n) \ar[r] & \Pois(f^*\cB_Y(\infty) \otimes_{\bD_{X_{DR}}(\infty)} \cB_X(\infty), n)
}$$
where $g^*_{\cB}(\infty)$ is the induced map
$$g^*_{\cB}(\infty) \colon f^*\cB_Y(\infty) \otimes_{\bD_{X_{DR}}(\infty)} \cB_X(\infty) \to \cB_X(\infty).$$
It follows that in order to prove the proposition, it will suffice to show that there is an equivalence
$$ f^*\cB_Y(\infty) \otimes_{\bD_{X_{DR}}(\infty)} \cB_X(\infty) \cong g^*\cB_{X \times Y}(\infty)$$
of $\bD_{X_{DR}}(\infty)$-modules. This can be checked directly: for every affine $A$, given an $A$-point of $X_{DR}$, the value of $g^*\cB_{X \times Y}(\infty)$ on $A$ is by definition $\bD((X \times Y)_A)(\infty)$, where $(X\times Y)_A$ is the fiber product
$$
\xymatrix{
(X\times Y)_A \ar[r] \ar[d]  & X  \ar[r] & X\times Y \ar[d] \\
\Spec A \ar[r]  & X_{DR} \ar[r] & X_{DR} \times Y_{DR}  \\ 
}
$$

But $(X \times Y)_A$ is naturally equivalent to $X_A \times_A Y_A$, so that 
$$\bD((X \times Y)_A)(\infty) \cong \bD(X_A)(\infty) \otimes_{\bD_{X_{DR}}(\infty)} \bD(Y_A)(\infty)$$
which concludes the proof.
\end{proof}

Notice that Theorem \ref{thm:graphpoisson} gives further examples of coisotropic structures: for every $n$-shifted Poisson derived Artin stack $X$, the map to $\pt=\Spec k$ is naturally a Poisson map, where $\Spec k$ is considered with its trivial $n$-Poisson structure. The graph of this map is the identity map on $X$, which therefore admits a canonical coisotropic structure, already constructed in Example \ref{ex:identity} above. Notice also that the space of Poisson maps $X\rightarrow \pt$ is equivalent to the space $\Pois(X, n)$ of $n$-shifted Poisson structures on $X$. We therefore get an equivalence $\Pois(X, n)\cong \Cois(\id, n)$ exactly as in Example \ref{ex:identity}.

The identity morphism $X\rightarrow X$ is also a Poisson morphism. Its graph is the diagonal $X\to X \times X$, which then admits a canonical coisotropic structure.

\section{Coisotropic intersections}
In this section we state and prove our second main result which extends the Lagrangian intersection theorem (see \cite[Theorem 2.9]{PTVV}) in the context of shifted Poisson structures.

\subsection{Affine case}

We begin with an affine version of the coisotropic intersection theorem. We consider the symmetric monoidal dg category $\cM$ as in \cite[Section 1.1]{MS1}.

\begin{prop}\label{prop:affineintersections}
Consider a diagram
$$\xymatrix @R=.5pc {
& A \ar[dr]^{g} \ar[dl]_{f}& \\
B_1&& B_2\\
}$$
in the $\infty$-category $\bCAlg(\cM)$ of commutative algebras in $\cM$. Suppose we are given coisotropic structures in $\Cois(f,n)$ and $\Cois(g,n)$, such that the $\bP_{n+1}$-structures on $A$ coincide. Then the tensor product $B_1 \otimes_A B_2$ carries a natural $\bP_n$-structure such that the map
$$B_1^{op} \otimes B_2 \longrightarrow B_1 \otimes_A B_2$$
is a Poisson morphism, where $B_1^{op}$ is the algebra $B_1$ taken with the opposite Poisson structure. 
\end{prop}

The proposition above is an easy consequence of Proposition \ref{prop:coiscorrespondences} below, which is a slightly more general result.

\begin{remark}
Note that \cite[Theorem 1.9]{Sa1} gives an alternative way to construct coisotropic intersections. Unfortunately, we do not know if the induced $\bP_n$ structures on $B_1\otimes_A B_2$ are equivalent.
\end{remark}

Consider two objects $A, B \in \balg_{\bP_{n+1}}(\cM)$. We say that a diagram in $\bCAlg(\cM)$ of the form
$$\xymatrix @R=.5pc {
A \ar[dr]&  & B \ar[dl] \\
&L& \\
}$$
is an \emph{affine $n$-shifted coisotropic correspondence from $A$ to $B$} if the induced map $A \otimes B^{op} \to L$ is endowed with an $n$-shifted coisotropic structure, relative to the given $n$-shifted Poisson structure on $A \otimes B^{op}$. Given such an affine $n$-shifted coisotropic correspondence we obtain a $\bP_n$-algebra structure on $L$.

\begin{prop}\label{prop:coiscorrespondences}
Let $A,B$ and $C$ be objects of $\balg_{\bP_{n+1}}(\cM)$, and suppose we are given affine $n$-shifted coisotropic correspondences $A \to L_1 \leftarrow B$ from $A$ and $B$ and $B \to L_2 \leftarrow C$ from $B$ to $C$. Let $L_{12} = L_1 \otimes_B L_2$. Then $A\to L_{12} \leftarrow C$ is an affine $n$-shifted coisotropic correspondence from $A$ to $C$ such that the projection $L_1\otimes L_2\rightarrow L_{12}$ is a morphism of $\bP_n$-algebras.
\end{prop}
\begin{proof}
Interpreting the given coisotropic structures as in Corollary \ref{cor:coiscptvv}, we can interpret
\[A,B,C \in \balg(\balg_{\bP_n}(\cM))\]
and
\[L_1 \in \bimod{A}{B}(\balg_{\bP_{n}}(\cM)),\qquad L_2 \in \bimod{B}{C}(\balg_{\bP_{n}}(\cM)).\]

Therefore, using composition of bimodules we see that $L_{12}\in\bimod{A}{C}(\balg_{\bP_n}(\cM))$, i.e. $L_{12}$ is an affine $n$-shifted coisotropic correspondence from $A$ to $C$. The last statement follows from the existence of a natural projection $L_1\otimes L_2\rightarrow L_1\otimes_B L_2$ in the $\infty$-category $\balg_{\bP_n}(\cM)$.
\end{proof}

\begin{remark}\label{rmk:coiscorrespondences}
The above Proposition \ref{prop:coiscorrespondences} gives a way to \emph{compose} affine coisotropic correspondences. Using the Poisson additivity proven in  \cite{Sa2} and the construction of the Morita $(\infty, m)$-category of $\bE_m$-algebras given in \cite{Ha} and \cite[Section 3]{Sch} it is indeed possible to construct an $(\infty,m)$-category whose objects are $\bP_n$-algebras in $\cM$, and whose $i$-morphisms are $i$-fold coisotropic correspondences. We will come back to these questions in a future work.
\end{remark}

\subsection{Intersections of derived stacks}

Notice that Proposition \ref{prop:affineintersections} recovers in particular the constructions in \cite{BG} for affine schemes. More generally, derived algebraic geometry provides a suitable general context to interpret the results of Baranovsky and Ginzburg: we now extend Proposition \ref{prop:affineintersections} to general derived stacks, giving a general conceptual explanation for the Gerstenhaber algebra structure constructed in \cite{BG}.

First, we need a lemma.
\begin{lm}
Consider a diagram of derived Artin stacks
$$
\xymatrix{
K \ar[r]^{\phi}  & X \ar[r]^i \ar[d]^j & Y \ar[d] \\
& Z \ar[r] & W 
}
$$
where the square on the right is Cartesian. Then the diagram of quasi-coherent sheaves on $K$
$$
\xymatrix{
\bT_{K/X} \ar[r] \ar[d] & \bT_{K/Y} \ar[d] \\
\bT_{K/Z} \ar[r] & \bT_{K/W} 
}
$$
is Cartesian.
\label{lm:cartesianT}
\end{lm}
\begin{proof}
From the diagram of stacks, one immediately gets two fiber sequences of quasi-coherent sheaves on $K$
$$
\xymatrix{
\bT_{K/Y} \ar[r]  & \bT_{K/W} \ar[r] & \phi^*i^*\bT_{Y/W}
}
$$
$$
\xymatrix{
\bT_{K/Z} \ar[r]  & \bT_{K/W} \ar[r] & \phi^*j^*\bT_{Z/W}
}
$$
and therefore the pullback of
$$
\xymatrix{
 & \bT_{K/Y} \ar[d] \\
\bT_{K/Z} \ar[r] & \bT_{K/W} 
}
$$
is equivalent to the fiber of the morphism $\bT_{K/W} \to \phi^*i^*\bT_{Y/W} \oplus \phi^*j^*\bT_{Z/W}$. But by general properties of Cartesian squares, $\bT_{X/W} \cong i^*\bT_{Y/W} \oplus j^*\bT_{Z/W}$, and hence we get that $\phi^*\bT_{X/W} \cong \phi^*i^*\bT_{Y/W} \oplus \phi^*j^*\bT_{Z/W}$. We conclude by observing that the fiber of the map
$$\bT_{K/W} \to \phi^*\bT_{X/W}$$
is equivalent to $\bT_{K/X}$.
\end{proof}

We have the following analogue of Proposition \ref{prop:affineintersections} for general derived stacks.
\begin{thm}\label{thm:generalintersections}
Consider a diagram
$$\xymatrix @R=.5pc {
L_1 \ar[dr]^{f}& & L_2\ar[dl]_{g} \\
&X&\\
}$$
of derived Artin stacks. Suppose we are given $n$-shifted coisotropic structures $\gamma_1 \in \Cois(f,n)$ and $\gamma_2 \in \Cois(g,n)$ on the morphisms $f$ and $g$, such that the $n$-shifted Poisson structures on $X$ coincide. Then the derived intersection $Y := L_1 \times_X L_2$ carries a natural $(n-1)$-shifted Poisson structure, such that the map
$$ Y \longrightarrow L_1 \times L_2$$
is a morphism of $(n-1)$-shifted Poisson derived stacks, where $L_1$ is equipped with the opposite Poisson structure. 
\end{thm}
\begin{proof}
The Cartesian diagram of stacks
$$
\xymatrix{
Y \ar^{j}[r] \ar^{i}[d] & L_1\ar^-{f}[d] \\
L_2 \ar^-{g}[r] & X
}
$$
induces a commutative square of $\bD_{Y_{DR}}(\infty)$-algebras
$$
\xymatrix{
j^*f^*\cB_X(\infty) \cong i^*g^*\cB_X(\infty) \ar[r] \ar[d] & j^*\cB_{L_1}(\infty) \ar[d] \\
i^*\cB_{L_2}(\infty) \ar[r] & \cB_Y(\infty)
}
$$

By definition the two coisotropic structures $\gamma_1$ and $\gamma_2$ produce two $\bP_{[n+1,n]}$-structures on the maps
$$j^*f^*\cB_X(\infty) \to j^*\cB_{L_1}(\infty) \text{ \ and \ } i^*g^*\cB_X(\infty) \to i^*\cB_{L_2}(\infty)$$
so that by Proposition \ref{prop:affineintersections} we obtain a natural $\bP_n$-structure on the coproduct
$$j^*\cB_{L_1}(\infty) \otimes_{i^*g^*\cB_X(\infty)} i^*\cB_{L_2}(\infty) \ .$$
Our goal is now to show that this coproduct is actually equivalent to $\cB_Y(\infty)$, which would immediately conclude the proof.
Notice that the twist by $k(\infty)$ commutes with colimits, so that is enough to show that
$$j^*\cB_{L_1} \otimes_{i^*g^*\cB_X} i^*\cB_{L_2} \cong \cB_Y$$
as $\bD_{Y_{DR}}$-algebras.

Let $\Spec A \to Y_{DR}$ be an $A$-point of $Y_{DR}$. We want to prove that $j^*\cB_{L_1} \otimes_{i^*g^*\cB_X} i^*\cB_{L_2}$ and $\cB_Y$ coincide on the point $\Spec A \to Y_{DR}$. By definition, the value of $\cB_Y$ on this point is $\bD(Y_A)$, where $Y_A$ is the perfect formal derived stack over $\Spec A$ constructed as the fiber product
$$
\xymatrix{
Y_A \ar[r] \ar[d] & Y \ar[d] \\
\Spec A \ar[r] &Y_{DR}
}
$$
Since the $(-)_{DR}$ construction is defined as a right adjoint, it automatically commutes with limits, so that $Y_{DR} \cong L_{1, DR} \times_{X_{DR}} L_{2, DR}$. In particular, any $A$-point of $Y_{DR}$ has corresponding $A$-points of $L_{1, DR}, L_{2, DR}$ and $X_{DR}$, for which one can define fibers $L_{1, A}, L_{2, A}$ and $X_A$. Therefore, we need to show that
$$\bD(Y_A) \cong \bD(L_{1, A}) \otimes_{\bD(X_A)} \bD(L_{2, A})$$
as graded mixed dg algebras.

We start by remarking that the fiber square
$$
\xymatrix{
Y_A \ar[r] \ar[d] & L_{1, A} \ar[d] \\
L_{2, A} \ar[r] & X_A
}
$$
induces a map of graded mixed cdgas
$$\bD(L_{1, A}) \otimes_{\bD(X_A)} \bD(L_{2, A}) \to \bD(Y_A)$$
by the universal property of the coproduct. In order to prove that this map is an equivalence, it is enough to check it at the level of commutative algebras, forgetting the graded mixed structures. But the forgetful functor
$$\bCAlg(\bdg_k^{gr,\epsilon}) \longrightarrow \bCAlg(\bdg_k)$$
comes by definition from the forgetful functor
$$\comod_B(\bdg_k) \longrightarrow \bdg_k$$
where $B$ is the bialgebra $B= k[t,t^{-1} ]\otimes_k k[x]$ and $\comod_B(\bdg_k)$ is the category of $B$-comodules in $\bdg_k$, as explained in \cite[Section 1.2]{MS1}. In particular, this means that forgetting the graded mixed structure preserves colimits, so that the underlying commutative algebra of the pushout of
$$
\xymatrix{
\bD(X_A) \ar[r] \ar[d] & \bD(L_{1, A}) \\
\bD(L_{2, A})& 
}
$$
is exactly the tensor product of commutative algebras $\bD(L_{1, A}) \otimes_{\bD(X_A)} \bD(L_{2, A})$. Since the stacks $X_A, L_{1, A}, L_{2, A}$ are all \emph{algebraisable} (in the sense of Section 2.2 of \cite{CPTVV}), by \cite[Theorem 2.2.2]{CPTVV} we have equivalences of commutative algebras
$$\bD(L_{1, A}) \otimes_{\bD(X_A)} \bD(L_{2, A}) \cong \Sym_{A^{red}}(\bL_{A^{red}/L_{1, A}}[-1])\otimes_{  \Sym_{A^{red}}(\bL_{A^{red}/X_A}[-1])} \Sym_{A^{red}}(\bL_{A^{red}/L_{2, A}}[-1])$$
$$\bD(Y_A) \cong  \Sym_{A^{red}}(\bL_{A^{red}/Y_A}[-1])$$

We can now just apply Lemma \ref{lm:cartesianT} to the diagram of algebraisable stacks
$$
\xymatrix{
\Spec(A^{red}) \ar[r]  & Y_A \ar[r] \ar[d] & L_{1, A} \ar[d] \\
& L_{2, A} \ar[r] & X_A 
}
$$
and get a Cartesian square of $A^{red}$-modules
$$
\xymatrix{
\bT_{A^{red}/Y_A} \ar[r] \ar[d] & \bT_{A^{red}/L_{1, A}} \ar[d] \\
\bT_{A^{red}/L_{2, A}} \ar[r] & \bT_{A^{red}/X_A} 
}
$$
From this we deduce a pushout diagram of $A^{red}$-algebras
$$
\xymatrix{
\Sym_{A^{red}}(\bL_{A^{red}/X_A}[-1]) \ar[r] \ar[d] & \Sym_{A^{red}}(\bL_{A^{red}/L_{1, A}}[-1]) \ar[d] \\
\Sym_{A^{red}}(\bL_{A^{red}/L_{2, A}}[-1]) \ar[r] & \Sym_{A^{red}}(\bL_{A^{red}/Y_A} [-1])
}
$$
which is exactly what we wanted.
\end{proof}

\begin{example}
Let $G$ be an affine algebraic group. Assume that $\B G$ carries a 1-shifted Poisson structure and the basepoint $\pt\rightarrow \B G$ carries a coisotropic structure. By Theorem \ref{thm:generalintersections} we obtain an ordinary Poisson structure on $G\cong \pt\times_{\B G} \pt$ which is easily seen to be multiplicative, i.e. $G$ carries a Poisson-Lie structure. It is shown in \cite[Corollary 2.11]{Sa3} that in fact the space of 1-shifted coisotropic structures on $\pt\rightarrow \B G$ is equivalent to the set of Poisson-Lie structures on $G$.
\end{example}

\begin{remark}
Following Remark \ref{rmk:coiscorrespondences}, one could generalize Proposition \ref{prop:coiscorrespondences} to possibly non-affine coisotropic correspondences: this will likely lead to a construction of the full $(\infty,m)$-category of coisotropic correspondences, even in the non-affine case. Namely, there exists an $(\infty, m)$-category whose objects are $n$-shifted Poisson stacks, and whose $i$-morphisms are $i$-fold coisotropic correspondences. This has to be considered as a derived incarnation of the so-called \emph{Poisson category} studied by Weinstein \cite{We}. We plan to give a detailed construction of this category in a future paper.
\end{remark}

\section{Non-degenerate coisotropic structures}
\label{sect:nondegenerate}

The purpose of this section is to introduce the notion of non-degeneracy of a coisotropic structure. This is a relative version of non-degenerate Poisson structures, as treated in \cite{CPTVV} or \cite{Pri1}. Our main result is a proof of \cite[Conjecture 3.4.5]{CPTVV}, stating that the space of non-degenerate coisotropic structures is equivalent to the space of Lagrangian structures in the sense of \cite{PTVV}.

\subsection{Definition of non-degeneracy}

We begin by first looking at the affine case. Recall the following notion from \cite[Corollary 1.4.24]{CPTVV}.
\begin{defn}\label{def:non-degpoisson}
Let $A$ be a commutative algebra in $\cM$ such that $\bL_A^{int}$ is a dualizable $A$-module. Suppose moreover that $A$ is equipped with an $n$-shifted Poisson structure. We say that it is \emph{non-degenerate} if the induced morphism
\[\pi^\sharp_A\colon \bL_A\longrightarrow \bT_A[-n]\]
is an equivalence.
\end{defn}
Suppose we have a cofibrant algebra $A \in \CAlg_M$. Then in this case we adopt the convention that $\pi^\sharp_A(f \ddr g) = \pm f [\pi_2, g]$, where $\pi_2$ is the underlying bivector of $\pi$, and $\pm$ is the Koszul sign.

Equivalently, as in \cite[Definition 1.4.18]{CPTVV}, a $\bP_{n+1}$-algebra $A$ is non-degenerate if the morphism
\[\bDR^{int}(A) \lra \bfPol^{int}(A,n)\]
induced by the Poisson bracket is an equivalence in $\cM^{gr}$.

We now deal with the case of relative Poisson algebras. Let $(A,B)$ be a $\bP_{[n+1,n]}$-algebra in $\cM$, and let $f:A \to B$ be the underlying morphism in $\bCAlg_{\cM}$. Using the description of $\bP_{[n+1, n]}$-structures in terms of relative polyvectors we see that the induced map
\[\bL_{B/A}[-1]\longrightarrow f^*\bL_A\xrightarrow{f^*\pi^\sharp_A} f^*\bT_A[-n]\longrightarrow \bT_{B/A}[-n+1]\]
is null-homotopic. Therefore, we get a morphism of fiber sequences of $B$-modules:
\begin{equation}
\xymatrix{
\bL_{B/A} [-1] \ar[r] \ar@{.>}[d] & f^*\bL_A \ar[r] \ar^{f^*\pi_A^\sharp}[d] & \bL_B \ar@{.>}[d] \\
\bT_B[-n] \ar[r] & f^*\bT_A[-n] \ar[r] & \bT_{B/A}[-n+1]. 
}
\label{eq:nondegeneratediagramaffine}
\end{equation}

Note that if $\pi_A^\sharp$ and one of the dotted maps are equivalences, so is the other dotted map.

\begin{defn}\label{def:non-degcoisotropics}
Let $f\colon A\rightarrow B$ be a morphism of commutative algebras in $\cM$ equipped with an $n$-shifted coisotropic structure.

\begin{itemize}
\item We say that the coisotropic structure is \emph{non-degenerate} if the $n$-shifted Poisson structure on $A$ is non-degenerate and one of the dotted maps in diagram \eqref{eq:nondegeneratediagramaffine} is an equivalence.

\item The space $\Cois^{nd}(f, n)$ of \emph{non-degenerate $n$-shifted coisotropic structures} on $f$ is the subspace of $\Cois(f, n)$ consisting of non-degenerate points.
\end{itemize}
\end{defn}

We will now generalize the notion of non-degeneracy to stacks. Let $f\colon L \to X$ be a morphism of derived Artin stacks. Suppose we are given an $n$-shifted coisotropic structure on $f$ in the sense of Definition \ref{def:generalcoisotropics}. This means that, in particular, we have a map of graded dg Lie algebras
$$k(2)[-1] \longrightarrow \bfPol(X,n)[n+1]$$
such that the induced map
$$k(2)[-1] \longrightarrow \bfPol(L/X,n-1)[n+1]$$
is homotopic to zero. Looking at weight 2 components, the shifted Poisson structure on $X$ induces by adjunction a morphism of perfect complexes on $X$
$$ \pi^{\sharp}_X\colon\ \bL_X \to \bT_X[-n]\ ,$$
and the coisotropic condition implies that the induced map $\bL_{L/X} \to \bT_{L/X}[-n+2]$ is homotopic to zero. This in turn implies the existence of dotted arrows in the diagram
\begin{equation}
\xymatrix{
\bL_{L/X} [-1] \ar[r] \ar@{.>}[d] & f^*\bL_X \ar[r] \ar[d]^{\pi^{\sharp}_X} & \bL_L \ar@{.>}[d] \\
\bT_L[-n] \ar[r] & f^*\bT_X[-n] \ar[r] & \bT_{L/X}[-n+1] 
}
\label{eq:nondegeneratediagramstacks}
\end{equation}
where both horizontal rows are fiber sequences of perfect complexes on $L$.

\begin{defn}\label{def:non-degcoisotropics2}
Let $f\colon L\rightarrow X$ be a morphism of derived Artin stacks equipped with an $n$-shifted coisotropic structure.

\begin{itemize}
\item We say that the coisotropic structure is \emph{non-degenerate} if the $n$-shifted Poisson structure on $X$ is non-degenerate and one of the dotted maps in diagram \eqref{eq:nondegeneratediagramstacks} is an equivalence.

\item We denote the subspace of non-degenerate points by $\Cois^{nd}(f, n)\subset \Cois(f, n)$.
\end{itemize}
\end{defn}

\begin{example}\label{ex:ndcois=lagr}
Suppose $i\colon L\hookrightarrow X$ is a smooth closed coisotropic subscheme of a smooth scheme $X$ carrying a ($0$-shifted) Poisson structure. Then the diagram \eqref{eq:nondegeneratediagramstacks} becomes
\[
\xymatrix{
\rN^*_{L/X} \ar[r] \ar@{.>}[d] & i^*\T^*_X \ar[r] \ar[d]^{\pi^\sharp_X} & \T^*_L \ar@{.>}[d] \\
\T^*_L \ar[r] & i^*\T_X \ar[r] & \rN_{L/X}
}
\]

The bivector $\pi_X$ is non-degenerate iff it underlines a symplectic structure. The coisotropic structure on $i$ is non-degenerate iff $\T^*_L\rightarrow \rN_{L/X}$ is an isomorphism which is the case precisely when $L\hookrightarrow X$ is Lagrangian.
\end{example}

By Theorem \ref{thm:spaceofcoisotropics2}, the datum of a coisotropic structure on $f\colon L\to X$ is equivalent to the datum of a $\bP_{n+1}$-structure on $\cB_X(\infty)$ and a compatible $n$-shifted coisotropic structure on the map $f^*\cB_X(\infty) \to \cB_L(\infty)$ in the category of $\bD_{L_{DR}}(\infty)$-modules.

\begin{cor}
Let $f\colon L \to X$ be a morphism of derived Artin stacks. An $n$-shifted coisotropic structure on $f$ is non-degenerate in the sense of Definition \ref{def:non-degcoisotropics2} if and only if the corresponding $n$-shifted Poisson structure on $\cB_X(\infty)$ and the $n$-shifted coisotropic structure on $f^*\cB_X(\infty) \to \cB_L(\infty)$ are both non-degenerate in the sense on Definitions \ref{def:non-degpoisson} and \ref{def:non-degcoisotropics}.
\end{cor}

This is an immediate consequence of the general correspondence between geometric differential calculus on derived stacks and algebraic differential calculus on the associated prestacks of Tate principal parts as described in \cite{CPTVV}.

Alternatively, one can avoid using twists by $k(\infty)$ as follows. Consider the prestack of graded mixed commutative algebras $f^*\cB_X$ and define the space of \emph{Tate $n$-shifted Poisson structures} on $f^*\cB_X$ to be 
$$\Pois^t(f^*\cB_X,n) := \Map_{\alg_{\Lie}^{gr}}(k(2)[-1], \bfPol^t(f^*\cB_X,n)[n+1]),$$
where $\bfPol^t(f^*\cB_X, n)$ is the Tate realization of the prestack of bigraded mixed $\bP_{n+2}$-algebras $\bfPol^{int}(f^*\cB_X, n)$. We have an equivalence of graded $\bP_{n+2}$-algebras
\[\bfPol^t(f^*\cB_X, n)\cong \bfPol(f^*\cB_X(\infty), n)\]
and hence an equivalence of spaces
\[\Pois^t(f^*\cB_X, n)\cong \Pois(f^*\cB_X(\infty), n).\]

Similarly, consider the map of prestacks of graded mixed algebras $f^*_{\cB}\colon f^*\cB_X \to \cB_L$,
and define the space of \emph{Tate $n$-shifted coisotropic structures} on $f^*_{\cB}$ to be
\[\Cois^t(f^*_\cB,n) := \Map_{\alg_{\Lie}^{gr}}(k(2)[-1], \bfPol^t(f^*_{\cB},n)[n+1]).\]
Again, we have an equivalence
\[\Cois^t(f^*_{\cB},n) \simeq \Cois(f^*_{\cB}(\infty),n).\]

We also have obviously defined subspaces of non-degenerate structures $\Pois^{t,nd}(f^*\cB_X,n)$ and $\Cois^{t,nd}(f^*_{\cB},n)$.
Notice that by definition we have a Cartesian square
\[
\xymatrix{
\Cois^{nd}(f,n) \ar[r] \ar[d] & \Cois^{t,nd}(f^*_{\cB},n) \ar[d] \\
\Pois^{t,nd}(\cB_X,n) \ar[r] & \Pois^{t,nd}(f^*\cB_X,n).
}
\]

\subsection{Symplectic and Lagrangian structures}

We recall the notions of shifted symplectic and shifted Lagrangian structures, defined and studied in \cite{PTVV} and \cite{CPTVV}.

\begin{defn}
Let $A \in \bCAlg_{\cM}$. The \emph{space of closed $2$-forms of degree $n$ on $A$} is
$$\mathcal A^{2,cl}(A,n):= \Map_{\dgmix}(k(2)[-1], \bDR(A)[n+1]).$$
\end{defn}
In particular, every closed 2-form $\omega$ of degree $n$ has an underlying 2-form $\omega_2 \in \Sym^2_A(\bL_A[-1])$. If the $A$-module $\bL_A$ is dualizable, this in turn gives rise to a morphism
$$\omega^{\sharp}\colon \bT_A \to \bL_A [n].$$

Explicitly, suppose $A$ is a cofibrant commutative algebra in $M$, and suppose the underlying 2-form is written as $\omega_2 = \sum_{i,j} \omega_{ij} \ddr a_i \ddr a_j$. Then our convention is that
\[\omega^\sharp(v) = \pm\sum_{i, j} 2\omega_{ij} [v, a_i] \ddr a_j,\]
where $v\in\bT_A$ and $\pm$ is the Koszul sign.

\begin{defn}
Let again $A \in \bCAlg_{\cM}$, and suppose moreover that the cotangent complex $\bL_A$ is dualizable.
\begin{itemize}
\item We say that a point $\omega \in \mathcal A^{2,cl}(A,n)$ is \emph{non-degenerate} if the above map $\omega^{\sharp}$ is an equivalence.
\item The space $\Symp(A,n)$ of \emph{$n$-shifted symplectic structures} on $A$ is the subspace of $\mathcal A^{2,cl}(A,n)$ of non-degenerate forms.
\end{itemize}
\end{defn}

Suppose now $f\colon A \to B$ is a morphism in $\bCAlg_{\cM}$. There is an induced map $\bDR(A)\to \bDR(B)$ of graded mixed cdgas and denote by $\bDR(f)$ the fiber of this map.
\begin{defn}
With notations as above, the space $\Isot(f,n)$ of \emph{$n$-shifted isotropic structures} is
$$\Isot(f,n) := \Map_{\dgmix}(k(2)[-1], \bDR(f)[n+1]).$$
\end{defn}
Informally, elements of $\Isot(f,n)$ are closed 2-forms of degree $n$ on $A$, whose restriction to $B$ is homotopic to zero. In other words, there is a fiber sequence of spaces
$$\Isot(f,n) \to \mathcal A^{2,cl}(A,n) \to \mathcal A^{2,cl}(B,n).$$

Let $\la \in \Isot(f,n)$, and suppose that $\bL_A$ and $\bL_B$ are both dualizable. The point $\la$ produces a map $\bT_A \to \bL_A[n]$ of $A$-modules, such that the composite
\[\bT_B\longrightarrow f^*\bT_A\xrightarrow{f^*\omega_A^\sharp} f^*\bL_A[n]\longrightarrow \bL_B[n]\]
is null-homotopic. This yields a diagram of $B$-modules
\begin{equation}
\xymatrix{
\bT_B \ar[r] \ar@{.>}[d] & f^*\bT_A \ar[r] \ar^{f^*\omega_A^\sharp}[d] & \bT_{B/A}[1] \ar@{.>}[d] \\
\bL_{B/A}[n-1] \ar[r] & f^*\bL_A[n] \ar[r] & \bL_B[n],
}
\label{eq:lagrangiandiagramaffine}
\end{equation}
where both rows are fiber sequences. As before, if $\omega_A^\sharp$ and one of the dotted maps are equivalences, so is the other dotted map.

\begin{defn}
Let $f\colon A\to B$ be a morphism in $\bCAlg_{\cM}$ and suppose both $\bL_A$ and $\bL_B$ are dualizable. \begin{itemize}
\item We say that a point $\la \in \Isot(f,n)$ is \emph{non-degenerate} if $\omega_A^\sharp$ and of the dotted maps in diagram \eqref{eq:lagrangiandiagramaffine} is an equivalence.

\item The space $\Lagr(f,n)$ of \emph{$n$-shifted Lagrangian structures} is the subspace of $\Isot(f,n)$ consisting of non-degenerate points.
\end{itemize}
\end{defn}

These algebraic notions can be used to introduce the concepts of symplectic and Lagrangian structures for general derived stacks.

\begin{defn}
Let $X$ be a derived Artin stack. The space $\Symp(X,n)$ of $n$-shifted symplectic structures on $X$ is
$$\Symp(X,n) := \Symp(\cB_X,n),$$
where we regard $\cB_X$ as a commutative algebra in the category of $\bD_{X_{DR}}$-modules.
\end{defn}

\begin{remark}
By \cite[Proposition 2.4.15]{CPTVV} this notion recovers the original global definition of an $n$-shifted symplectic structure given by \cite[Definition 1.18]{PTVV}.
\end{remark}

The following is \cite[Theorem 3.2.4]{CPTVV} and \cite[Theorem 3.33]{Pri1}.
\begin{thm}\label{thm:poisnd=symp}
Let $X$ be a derived Artin stack. There is an equivalence of spaces
\[ \Pois^{nd}(X) \simeq \Symp(X,n). \]
\end{thm}

In the relative case the definitions are analogous. Recall that given a map $f\colon L \to X$ of derived stacks we have an induced map
$$f^*_{\cB}\colon f^*\cB_X \to \cB_L$$
of commutative algebras in the category of $\bD_{L_{DR}}$-modules. Notice that any shifted symplectic structure on $X$ gives in particular a shifted symplectic structure on $f^*\cB_X$; in other words, there is a natural map of spaces
$$\Symp(X,n) \simeq \Symp(\cB_X,n) \longrightarrow \Symp(f^*\cB_X,n).$$

\begin{defn}
Let $f\colon L \to X$ be a map of derived Artin stacks. The space $\Lagr(f,n)$ of \emph{$n$-shifted Lagrangian structures} on $f$ is given by the pullback
\[
\xymatrix{
\Lagr(f, n) \ar[r] \ar[d] & \Lagr(f^*_{\cB},n) \ar[d] \\
\Symp(X,n) \ar[r] & \Symp(f^*\cB_X,n).
}
\]
\end{defn}

\subsection{Compatible pairs}\label{sect:compatiblepairs}

The remainder of the section is devoted to proving a derived analogue of Example \ref{ex:ndcois=lagr}.
In particular, we would like to compare the spaces $\Cois^{nd}(f,n)$ and $\Lagr(f,n)$. With this goal in mind, we begin by developing a general formalism of compatibility for pairs $(\ga,\la)$ formed by a coisotropic and a Lagrangian structure in a general symmetric monoidal $\infty$-category, following the approach of \cite{Pri4}. The notion of compatibility of an $n$-shifted Poisson and an $n$-shifted symplectic structure has previously appeared in \cite[Definition 1.4.20]{CPTVV} and \cite[Definition 1.20]{Pri1}.

Consider a map $f\colon A \to B$ in $\CAlg_{M}$, such that both $A$ and $B$ are cofibrant, and consider the graded homotopy $\bP_{[n+2,n+1]}$-algebra
$$(\bfPol(A,n), \bfPol(B/A,n-1))$$
constructed in \cite{MS1}. We pick a \emph{strict} graded $\bP_{[n+2,n+1]}$ algebra, which is equivalent to the latter, and we denote it by
\[ (\Pol(A,n), \Pol(B/A,n-1)). \]
We denote by $\Pol(f,n)$ the homotopy fiber of the underlying morphism of graded commutative algebras. Consider an element $\gamma \in \Pol(f,n)^{\geq 2}[n+2]$. By the results of \cite[Section 3.6]{MS1}, we know that $\gamma$ induces a pair of compatible $k$-linear derivations $\phi_{\gamma,A}$ and $\phi_{\gamma,B}$ on $\Pol(A,n)$ and $\Pol(B/A,n-1)$ respectively. Restricted to weight zero they fit in the diagram
\[
\xymatrix{
\Omega^1_A[-1] \ar[r] \ar^{\phi_{\gamma, A}}[d] & \Omega^1_B[-1] \ar^{\phi_{\gamma, B}}[d] \\
\Pol(A,n) \ar[r] & \Pol(B/A,n-1).
}
\]

Using the universal property of the symmetric algebra we obtain a diagram of commutative algebras
\[
\xymatrix{
\DR(A) \ar[r] \ar[d]^{\mu(-,\ga)_{A}} & \DR(B) \ar[d]^{\mu(-,\ga)_{B}} \\
\Pol(A,n) \ar[r] & \Pol(B/A,n-1).
}
\]

\begin{remark}
The construction above also applies to the general case of coisotropic structure on a morphism of derived Artin stacks. For example, let $L\hookrightarrow X$ be a smooth coisotropic submanifold of a Poisson manifold. Then Oh--Park \cite{OP} and Cattaneo--Felder \cite{CF} construct a certain homotopy $\bP_1$-algebra which as a graded commutative algebra coincides with $\Gamma(L, \Sym(\rN_{L/X}[-1]))$. We expect that it coincides with $\bfPol(L/X, -1)$ with the differential twisted by $\phi_{\gamma, L}$.
\end{remark}

The vertical arrows are the identity on weight 0 while their value on weight 1 generators is given by
$$\mu(a\ddr x, \ga)_A = a\phi_{\gamma,A}(x), \ \ \mu(b\ddr y, \ga)_B= b\phi_{\gamma,B}(y),$$
where $a, x \in A$ and $b,y \in B$. Moreover, if $\gamma$ satisfies the Maurer--Cartan equation, i.e. it corresponds to an $\infty$-morphism $k(2)[-1]\rightarrow \Pol(f, n)[n+1]$ of graded Lie algebras, the above diagram is a diagram of weak graded mixed algebras.

In particular, observe that every $\ga \in \Pol(f,n)[n+2]$ induces a map 
\[ \mu(-,\ga)\colon \DR(f) \to \Pol(f,n) \]
of commutative algebras, where $\DR(f)$ is the homotopy fiber of $\DR(A) \to \DR(B)$. Moreover, if $\ga$ defines a coisotropic structure then $\mu(-,\ga)$ is a map of graded mixed commutative algebras. 

Consider now the commutative dg algebra $k[\e]$, where $\e$ is of degree $0$ and satisfies $\e^2=0$. If $\g$ is a graded dg Lie algebra, then $\g \otimes k[\e]$ is still a graded dg Lie algebra, and there is a natural projection of graded Lie algebras $\g \otimes k[\e] \to \g$ sending $\e$ to zero. Let $\sigma\colon \g\rightarrow \g$ be the operator given by $p-1$ in degree $p$. Then $\id + \epsilon\sigma$ gives a section $\g\rightarrow \g \otimes k[\epsilon]$.

The projection $\g\otimes k[\e] \to \g$ induces a morphism of spaces
\[\Map_{\alg_{\Lie}^{gr}}(k(2)[-1], \g\otimes k[\e]) \to \Map_{\alg_{\Lie}^{gr}}(k(2)[-1], \g),\]
whose fibers have the following nice characterization.

\begin{lm}\label{lm:MCfiberk[e]}
The fiber of the map 
\[ \Map_{\alg_{\Lie}^{gr}}(k(2)[-1], \g\otimes k[\e]) \to \Map_{\alg_{\Lie}^{gr}}(k(2)[-1], \g),  \]
at a point corresponding to a Maurer--Cartan element $x\in \g^{\geq 2}$ is given by $\Map_{\dgmix}(k(2)[-1],\g_x)$, where $\g_x$ is the graded module $L$ equipped with the mixed structure $[x, -]$.
\end{lm}

The proof of this lemma is a straightforward computation, and we omit it.
Following \cite{Pri1}, it is convenient to introduce an auxiliary space.

\begin{defn}
Let again $f\colon A \to B$ be a morphism of commutative algebras in $M$. The \emph{tangent space of $n$-shifted coisotropic structures on $f$} is
\[ \TCois(f, n) = \Map_{\alg_{\Lie}^{gr}}(k(2)[-1], \Pol(f,n)[n+1]\otimes k[\e]). \]
\end{defn}

In particular, we get a natural map
\[ \id + \e \si\colon \Cois(f,n) \lra \TCois(f,n).\]
By Lemma \ref{lm:MCfiberk[e]} a point of $\TCois(f, n)$ is given by an $n$-shifted coisotropic structure $\gamma$ on $f$ together with a morphism of graded mixed complexes $k(2)[-1]\rightarrow \Pol_\gamma(f, n)[n+1]$, where $\Pol_\gamma(f, n)[n+1]$ is $\Pol(f, n)[n+1]$ equipped with the mixed structure $[\gamma, -]$.

Consider now the natural projection
\[ \TCois(f,n) \times \Lagr(f,n) \to \Cois(f,n) \times \Lagr(f,n), \]
which simply forgets the $\e$-component on $\TCois(f,n)$. The above constructions provide a section of this map.

\begin{lm}
The section
\[
\xymatrix @C=0.5pc @R=.3pc{
\Phi\colon \Cois(f,n) \times \Lagr(f,n) & \lra & \TCois(f,n) \times \Lagr(f,n) \\
(\gamma, \lambda) & \longmapsto & (\gamma + \epsilon(\sigma(\gamma) - \mu(\lambda, \gamma)), \lambda).
}
\]
is well-defined.
\end{lm}
\begin{proof}
We have to check that if $\gamma$ is a Maurer--Cartan element and $\lambda$ is a closed element, then $\gamma + \epsilon(\sigma(\gamma) - \mu(\gamma, \lambda))$ is also a Maurer--Cartan element.

For this it is enough to show that $a_1 = \gamma + \epsilon\sigma(\gamma)$ and $a_2 = \gamma - \epsilon \mu(\lambda, \gamma)$ separately satisfy the Maurer--Cartan equations. Indeed, $\id + \epsilon\sigma$ is a morphism of Lie algebras and hence sends Maurer--Cartan elements to Maurer--Cartan elements. For the second expression we can compute
\[[a_2, a_2] = [\gamma - \epsilon \mu(\lambda, \gamma), \gamma - \epsilon \mu(\lambda, \gamma)] = [\gamma, \gamma] - 2\epsilon [\gamma, \mu(\lambda, \gamma)].\]

Since $\mu(-, \gamma)$ is a morphism of weak graded mixed commutative algebras (see \cite[Definition 1.11]{MS1}), we have $[\gamma, \mu(\lambda, \gamma)] = \mu(\ddr \lambda, \gamma)$. Therefore,
\[\d a_2 + \frac{1}{2}[a_2, a_2] = \d\gamma + \frac{1}{2}[\gamma, \gamma] - \epsilon \mu((\d + \ddr)\lambda, \gamma)=0.\]
\end{proof}

We will repeatedly consider the following construction. Suppose $p\colon T\rightarrow S$ is a morphism of simplicial sets equipped with two sections $s\colon S\rightarrow T$ and $0\colon S\rightarrow T$. The section $0$ will be implicit and we encode the rest of the data in the following diagram:
\[
\xymatrix{
T \ar_{p}[r] & S \ar@/_1pc/_{s}[l]
}
\]

The following is introduced in \cite[Definition 1.23]{Pri1}.

\begin{defn}
The \emph{vanishing locus} of the diagram
\[
\xymatrix{
T \ar_{p}[r] & S \ar@/_1pc/_{s}[l]
}
\]
is defined to be the homotopy limit of
\[
\xymatrix{
S \ar@<.5ex>^{s}[r] \ar@<-.5ex>_{0}[r] & T \ar^{p}[r] & S
}
\]
\end{defn}

In other words, the vanishing locus of such a diagram parametrizes points $x\in S$ together with a homotopy $s(x)\sim 0(x)$ in $p^{-1}(x)$. Given this definition, we can define compatibility between coisotropic and Lagrangian structures.

\begin{defn}\label{def:comp}
We define the space $\Comp(f,n)$ of \emph{compatible pairs} to be the vanishing locus of
\[
\xymatrix{
\TCois(f,n) \times \Lagr(f,n) \ar[r] & \Cois(f,n) \times \Lagr(f,n) \ar@/_1.5pc/_{\Phi}[l]
}
\]
\end{defn}

In other words, elements of $\Comp(f,n)$ are given by pairs $(\gamma, \lambda)$ of a coisotropic and a Lagrangian structure on $f$ together with a homotopy from $\mu(\lambda, \gamma)$ to $\sigma(\gamma)$ in $\Pol_\gamma(f, n)$.

\begin{defn}
The space of \emph{non-degenerate compatible pairs} $\Comp^{nd}(f, n)$ is defined to be the homotopy fiber product
$$\xymatrix{
\Comp^{nd}(f,n) \ar[r]\ar[d] & \Comp(f,n) \ar[d] \\
\Cois^{nd}(f,n) \ar[r] & \Cois(f,n).
}$$
\end{defn}

In particular, the space $\Comp^{nd}(f,n)$ comes equipped with two projections to $\Cois^{nd}(f,n)$ and $\Lagr(f,n)$, giving a correspondence
\[
\xymatrix{
& \Comp^{nd}(f, n) \ar[dl] \ar[dr] & \\
\Cois^{nd}(f,n) && \Lagr(f,n).
}
\]

\begin{prop}\label{prop:comp=cois}
The map
$$\Comp^{nd}(f,n) \to \Cois^{nd}(f,n)$$
is an equivalence.
\end{prop}
\begin{proof}
If $\ga$ is a non-degenerate coisotropic structure, then $\mu(-,\ga)\colon \DR(f)\rightarrow \Pol_\gamma(f, n)$ is an equivalence. In particular, the space of two-forms $\lambda\in\DR(f)$ such that $\mu(\lambda, \gamma)\sim \sigma(\gamma)$ is contractible.
\end{proof}

It follows that given any morphism $f\colon A \to B$ in $\CAlg_{M}$, we get a map of spaces
$$\Cois^{nd}(f,n) \to \Lagr(f,n).$$
In particular, consider a map $f\colon L \to X$ of derived stacks. We have an induced map $f^*_{\cB}\colon f^*\cB_X \to \cB_L$ of $\bD_{L_{DR}}$-algebras; twisting by $k(\infty)$, we get $f^*_{\cB}(\infty) : f^*\cB_X(\infty) \to \cB_L(\infty)$ of commutative algebras in the category of graded mixed $\bD_{L_{DR}}(\infty)$-modules. From the above discussion, we know that there is a morphism
$$\Cois^{t,nd}(f^*_{\cB},n) \simeq \Cois^{nd}(f^*_{\cB}(\infty),n) \to \Lagr(f^*_{\cB}(\infty),n) \simeq \Lagr(f^*_{\cB},n).$$
Together with Theorem \ref{thm:poisnd=symp}, this produces a map
$$\Cois^{nd}(f,n) \to \Lagr(f,n).$$
The following is the main result of this section, which has been stated as a conjecture in \cite[Conjecture 3.4.5]{CPTVV}.

\begin{thm}\label{thm:coisnd=lagr}
Let $f\colon L \to X$ be a map of derived Artin stacks. Then the above map
$$\Cois^{nd}(f,n) \to \Lagr(f,n)$$
is an equivalence.
\end{thm}
Note that a treatment of this result in the case $n=0$ was given in \cite{Pri4} and our proof is a slight generalization of that.

\subsection{Reduction to graded mixed algebras}

Our first step in the proof of Theorem \ref{thm:coisnd=lagr} is to notice that we can reduce the problem to an algebraic question: in fact, by Theorem \ref{thm:poisnd=symp}, we only need to show that the map
$$\Cois^{t,nd}(f^*_{\cB},n) \to \Lagr(f^*_{\cB},n)$$
is an equivalence.

We claim that both $\Cois^{t,nd}(f^*_{\cB},n)$ and $\Lagr(f^*_{\cB},n)$ can obtained as global sections of prestacks.
Consider the internal relative Tate polyvectors
\[ ( \bfPol^{int,t}(f^*\cB_X,n) , \bfPol^{int,t}(\cB_L/f^*\cB_X,n-1) ). \]
This is a graded $\bP_{[n+2,n+1]}$-algebra in the $\infty$-category $\Fun((\dAff/L_{DR})^{op}, \dg_k)$,
or equivalently a diagram of graded $\bP_{[n+2,n+1]}$-algebras in cochain complexes. Let $\bfPol^{int,t}(f^*_{\cB},n)$ be the fiber of the underlying map of commutative algebras. Then there is an equivalence
\[	\Cois^t(f^*_{\cB},n) \simeq \Map_{\alg^{gr}_{\Lie}}(k(2)[-1], \Pol^t(f^*_{\cB},n)), \]
where $\Pol^t(f^*_{\cB},n)$ is the realization (i.e. the limit) of the diagram $\bfPol^{int,t}(f^*_{\cB},n)$.

Now define the prestack
$$\begin{array}{cccc}
\underline{\Cois}^t(f^*_{\cB},n)\colon & (\dAff/L_{DR})^{op} & \longrightarrow & \sSet \\
& (\Spec A \to L_{DR}) & \longmapsto & \Map_{\alg_{\Lie}^{gr}}(k(2)[-1], \bfPol^{int,t}(f^*_{\cB},n)(A)).
\end{array}$$
Consider the sub-prestack of $\underline{\Cois}^t(f^*_{\cB},n)$  of non-degenerate coisotropic structures, and denote it by $\underline{\Cois}^{t,nd}(f^*_{\cB},n)$.
By definition, by taking global sections (that is to say, taking the realization) of the prestack $\underline{\Cois}^{t,nd}(f^*_{\cB})$ we get the space $\Cois^{t, nd}(f^*_{\cB},n)$.

Similarly, we can define the prestack
$$\begin{array}{cccc}
\underline{\Isot}\colon & (\dAff/L_{DR})^{op} & \to & \sSet \\
& (\Spec A \to L_{DR}) & \mapsto & \Map_{\dgmix}(k(2)[-n-2], \bDR^{int}(f^*_{\cB})(A))
\end{array}$$
of $n$-shifted isotropic structures on $f^*_{\cB}$, and define $\underline{\Lagr}(f^*_{\cB},n)$ to be the sub-prestack of $\underline{\Isot}(f^*_{\cB},n)$ of Lagrangian structures. Once again, global sections of $\underline{\Lagr}(f^*_{\cB},n)$ are identified with the space $\Lagr(f^*_{\cB},n)$.

The construction of Section \ref{sect:compatiblepairs} produces a map of prestacks
\begin{equation}\label{eq:cois->lagr}
\underline{\Cois}^{t,nd}(f^*_{\cB},n) \longrightarrow \underline{\Lagr}(f^*_{\cB},n).
\end{equation}
Taking global sections on both sides we get back $\Cois^{t,nd}(f^*_{\cB},n) \to \Lagr(f^*_{\cB},n)$.

It follows that in order to prove Theorem \ref{thm:coisnd=lagr}, it is enough to show that the above map (\ref{eq:cois->lagr}) is an equivalence.
This can be checked object-wise; in particular, we need to show that for every $\Spec A \to L_{DR}$, the induced map
$$\Cois^{t,nd}(f^*_{\cB}(A),n) \to \Lagr(f^*_{\cB}(A),n)$$
is an equivalence.
Notice that in this case, $f^*_{\cB}(A)$ is a map between graded mixed cdgas. In particular, Theorem \ref{thm:coisnd=lagr} will follow from the following result.
\begin{thm}\label{thm:coisnd=lagr2}
Let $M$ be the model category of graded mixed complexes, and let $f\colon A \to B$ be a morphism in $\CAlg_{M}$. Then the map
$$\Cois^{t,nd}(f,n) \to \Lagr(f,n)$$
is an equivalence.
\end{thm}
\begin{remark}
By the correspondence between Tate realizations and twists by $k(\infty)$, the above theorem says that the map
\[ \Cois^{nd}(f(\infty),n) \to \Lagr(f(\infty),n) \]
is an equivalence,
where $f(\infty)\colon A(\infty) \to B(\infty)$ is the twist of $f$ by $k(\infty)$. 
\end{remark}

The proof of Theorem \ref{thm:coisnd=lagr2} occupies the rest of this section.

\subsection{Filtrations}

In order to make notations a bit simpler, we will omit the underlying map $f$ and the shift $n$ in what follows. Moreover, we will also remove the reference to Tate realization. In particular, the space of Tate $n$-coisotropic structures $\Cois^t(f,n)$ will be denoted $\Cois$, and similarly for $\Cois^{nd}, \TCois, \TCois^{nd}, \Lagr$.

Notice that the weight grading of $\Pol^t(f,n)$ automatically defines a filtration on it: we denote by $\Pol^t(f,n)^{\leq p}$ the graded Poisson algebra
\[ \Pol^t(f,n)^{\leq p} = \bigoplus_{i \leq p} \Pol^t(f,n)^i\]
where $\Pol^t(f,n)^i$ is the weight $i$ part of $\Pol^t(f,n)$. Note that $\Pol^t(f, n)^{\leq p}$ is naturally a quotient of $\Pol^t(f, n)$. This in turn induces a filtration on the space $\Cois$, and we set
$$\Cois^{\leq p} : = \Map_{\Lie^{gr}}(k(2)[-1], \Pol^t(f,n)[n+1]^{\leq p}),$$
and similarly $\Cois^{nd, \leq p}$ is the subspace of $\Cois^{\leq p}$ of non-degenerate coisotropic structures. Note that non-degeneracy is merely a condition on the underlying bivector. We define in a similar way the spaces $\TCois^{\leq p}$ and $\TCois^{nd,\leq p}$.

The same construction also applies to the space $\Lagr$. In fact, $\Lagr$ is by definition the subspace of
$$\Map_{\dgmix}(k(2)[-n-2], \DR(f))$$
given by non-degenerate maps. But using the weight grading on $\DR(f)$, one can define graded mixed modules $\DR(f)^{\leq p}$, a quotient of $\DR(f)$, and set $\Lagr^{\leq p}$ to be the subspace of
$$\Map_{\dgmix}(k(2)[-n-2], \DR(f)^{\leq p})$$
given by non-degenerate maps.

Notice that all our previous constructions of $\mu$ and $\si$ in Section \ref{sect:compatiblepairs} are compatible with these filtrations and
$$\TCois^{\leq p} \times \Lagr^{\leq p} \to \Cois^{\leq p} \times \Lagr^{\leq p}$$
admits a natural section $\Phi_p$. We define $\Comp^{\leq p}$ to be the vanishing locus of the section $\Phi_p$.

In particular, for every $p$ we have a commutative square
$$\xymatrix{
\Comp^{\leq p+1} \ar[r] \ar[d] & \Lagr^{\leq p+1} \ar[d] \\
\Comp^{\leq p} \ar[r] & \Lagr^{\leq p}.
}$$
In order to prove that the map $\Comp^{nd} \to \Lagr$ is an equivalence, it is enough to show that $\Comp^{nd,\leq p} \to \Lagr^{\leq p}$ is an equivalence for every $p \geq 2$. We will thus proceed by induction on $p$.

First, let us unpack the compatibility between coisotropic and Lagrangian structures.
\begin{lm}
Suppose $\gamma$ is an $n$-shifted coisotropic structure on $f$ and $\lambda$ an $n$-shifted Lagrangian structure on $f$ inducing morphisms
\[\gamma_A^\sharp\colon \bL_A\rightarrow \bT_A[-n],\qquad \gamma_B^\sharp\colon \bL_B\rightarrow \bT_{B/A}[1-n]\]
and
\[\lambda_A^\sharp\colon \bT_A\rightarrow \bL_A[n],\qquad \lambda_B^\sharp\colon \bT_{B/A}\rightarrow \bL_B[n-1].\]

The compatibility between $\gamma$ and $\lambda$ in weight $2$ is equivalent to the relations
\[\gamma_A^\sharp\circ \lambda_A^\sharp\circ \gamma_A^\sharp\cong \gamma_A^\sharp,\qquad \gamma_B^\sharp\circ \lambda_B^\sharp\circ \gamma_B^\sharp\cong \gamma_B^\sharp.\]
\label{lm:compatibilityformula}
\end{lm}
\begin{proof}
Suppose
\[\lambda_A = \sum_{i, j} \omega_{ij} \ddr a_i \ddr a_j\]
is the two-form on $A$ underlying $\lambda$ and $\gamma_A$ is the bivector on $A$ underlying $\gamma$. By definition
\[\mu(\lambda, \gamma)_A = \pm\sum_{i, j} \omega_{ij} [\gamma_A, a_i] [\gamma_A, a_j]\]
and the compatibility relation is that this bivector is homotopic to $\gamma_A$ itself. Since $\bL_A$ is dualizable, we can identify bivectors with antisymmetric maps $\bL_A\rightarrow \bT_A$. Let us now compute the induced maps. Suppose $a'\in A$. Then
\begin{align*}
[\mu(\lambda_A, \gamma_A), a'] &= \pm\sum_{i, j} \omega_{ij} [[\gamma_A, a_i] [\gamma_A, a_j], a'] \\
&= \pm \sum_{i, j} 2\omega_{ij} [\gamma_A, a_i] [[\gamma_A, a_j], a'] \\
&= \pm \sum_{i, j} 2\omega_{ij} [\gamma_A, a_i] [[\gamma_A, a'], a_j].
\end{align*}

But we have
\[(\lambda_A^\sharp\circ \gamma_A^\sharp)(a') = \pm \sum_{i, j} 2\omega_{ij} [[\gamma_A, a'], a_i] \ddr a_j\]
and thus the compatibility relation boils down to a homotopy
\[\gamma_A^\sharp\circ \lambda_A^\sharp\circ \gamma_A^\sharp\cong \gamma_A^\sharp.\]

For the second statement observe that we have a commutative diagram
\[
\xymatrix{
\DR(f)^{\geq 2}[n+2] \ar^{\mu(-, \gamma)}[d] \ar[r] & \bL_{B/A}\otimes \bL_B[n-1] \ar^{\widetilde{\gamma}_B^\sharp\otimes \gamma_B^\sharp}[d] \\
\Pol^t(f, n)^{\geq 2}[n+2] \ar[r] & \bT_B\otimes \bT_{B/A}[1-n]
}
\]
where $\widetilde{\gamma}_B^\sharp\colon \bL_{B/A}\rightarrow \bT_B[1-n]$ is the dual of $\gamma_B^\sharp$. Proceeding as before we see that
\[(\widetilde{\gamma}_B^\sharp\otimes \gamma_B^\sharp)(\lambda_B^\sharp)\cong \gamma_B^\sharp\circ \lambda_B^\sharp\circ \gamma_B^\sharp\]
where we have identified $\bL_{B/A} \otimes \bL_B \cong \Hom_B(\bT_{B/A}, \bL_B)$ and  $\bT_B\otimes \bT_{B/A}\cong \Hom_B(\bL_B, \bT_{B/A})$. Therefore, we obtain that the compatibility of $\gamma_B$ and $\lambda_B$ is equivalent to the relation
\[\gamma_B^\sharp\circ \lambda_B^\sharp\circ \gamma_B^\sharp\cong \gamma_B^\sharp.\]
\end{proof}

The base of the induction is given by the following statement.
\begin{prop}\label{prop:equivweight2}
The projection
$$\Comp^{nd, \leq 2} \to \Lagr^{\leq 2}$$
is an equivalence.
\end{prop}
\begin{proof}
The points of the space $\Comp^{nd,\leq2}$ are in particular pairs $(\ga,\la)$, where $\ga$ is a weight 2 cocycle in $\Pol^t(f,n)[n+2]$ and $\la$ is a weight 2 cocycle in $\DR(f)[n+2]$. Moreover, $\ga$ is non-degenerate, in the sense that the induced vertical morphisms
$$\xymatrix{
\bL_A \ar[r] \ar[d]^{\ga_A^{\sharp}} & \bL_B \ar[d]^{\ga_B^{\sharp}} \\
|\bT_A^{int}[-n]|^t \ar[r] & |\bT_{B/A}^{int}[-n+1]|^t
}
$$
are equivalences. Similarly, $\la$ defines a Lagrangian structure, so that we also have vertical morphisms
$$\xymatrix{
|\bT_A^{int}[-n]|^t \ar[r] \ar[d]^{\la_A^{\sharp}} & |\bT_{B/A}^{int}[-n+1]|^t \ar[d]^{\la_B^{\sharp}}\\
\bL_A \ar[r]  & \bL_B 
}
$$
which are again equivalences.

By Lemma \ref{lm:compatibilityformula} we have
\[\gamma_A^\sharp\circ \lambda_A^\sharp\circ \gamma_A^\sharp\cong \gamma_A^\sharp,\qquad \gamma_B^\sharp\circ \lambda_B^\sharp\circ \gamma_B^\sharp\cong \gamma_B^\sharp\]
and since both $\gamma_A^\sharp$ and $\gamma_B^\sharp$ are equivalences, we see that $\lambda_A^\sharp$ and $\lambda_B^\sharp$ are uniquely determined to be the inverses of $\gamma_A^\sharp$ and $\gamma_B^\sharp$.
\end{proof}

\subsection{Obstructions}

For the inductive step, suppose we are given a non-degenerate compatible pair $(\ga, \la) \in \Comp^{nd, \leq p}$. We start by studying the obstruction to extend it to something in $\Comp^{nd, \leq p+1}$. For this purpose let us define the obstruction spaces
\begin{align*}
\Obs(p+1, \Lagr) &= \DR(f)^{p+1}[n+3] \\        
\Obs(p+1, \Cois) &= \Pol^t(f,n)^{p+1}[n+3] \\
\Obs(p+1, \TCois) &= \Pol^t(f,n)^{p+1}[n+3] \otimes k[\e],
\end{align*}
where we apply the Dold--Kan correspondence to the complexes on the right to turn a complex into a simplicial set.

We denote by $\uObs(p+1, -)$ the corresponding trivial bundles; for example,
$$\uObs(p+1, \Lagr) = \Obs(p+1, \Lagr)\times \Lagr^{\leq p}.$$

This obstruction bundle has a natural non-trivial section. Suppose $\la \in \DR(f)^{\leq p}$ defines a Lagrangian structure in $\Lagr^{\leq p}$. Then we can also consider $\la$ as living in the whole $\DR(f)$, and take $\ddr \la$.
We define the map
$$ \Lagr^{\leq p}\to \Obs(p+1, \Lagr)$$
by sending $\la$ to $P_{p+1}\ddr \la$, where $P_{p+1}$ is the projection to the weight $(p+1)$-component of $\DR(f)$ and $\ddr$ is the mixed structure on $\DR(f)$. More explicitly, if the weight component of $\la$ are
$\la_2 + \dots + \la_p$, then $\la$ is sent to $\ddr \la_p$. In particular, this gives a section $s\colon \Lagr^{\leq p}\rightarrow \uObs(p+1, \Lagr)$.

The following lemma explains why we think of $\Obs(\Lagr,p+1)$ as obstruction spaces. 
\begin{lm}\label{lm:lagrobs}
There is a natural equivalence between $\Lagr^{\leq(p+1)}$ and the vanishing locus of
\[
\xymatrix{
\uObs(p+1, \Lagr) \ar[r] & \Lagr^{\leq p}. \ar@/_1pc/_{s}[l]
}
\]
\end{lm}

\begin{proof}
Let $\DR(f)^{[2,p]}$ be the sub-complex of $\DR(f)$ obtained by considering only weights in the interval $[2,p]$. In other words, $\DR(f)^{[2,p]}$ is the complex
$$\DR(f)^{[2,p]} = \bigoplus_{i=2}^{p} \DR(f)^i$$
equipped with the differential $d + \ddr$.

By definition, $m$-simplices of $\Lagr^{\leq p+1}$ are identified with closed elements in the complex
$$\Omega^\bullet(\Delta^m)\otimes \DR(f)^{[2, p+1]}[n+2].$$

On the other hand, the vanishing locus of the section $s$ has $m$-simplices given by closed elements $\la_2 + \dots +\la_p$ in
$$\Omega^\bullet(\Delta^m)\otimes \DR(f)^{[2, p]}[n+2].$$
together with an element $\la_{p+1}$ in $\Omega^\bullet(\Delta^m)\otimes \DR(f)(p+1)[n+2]$ such that $\ddr \la_p + \d \la_{p+1} = 0$. Therefore, they are identified with elements $\la_2 + \dots + \la_{p+1}$ of $\Lagr^{p+1}$. 
\end{proof}

We can also obtain a similar interpretation of the obstruction spaces for $\Cois$ and $\TCois$. More specifically, define a section
$$\Cois^{\leq p} \to \Obs(\Cois,p+1)$$
of the obstruction bundle for $\Cois$ by sending an element $\ga \in \Cois^{\leq p}$ to $\frac 12 P_{p+1} [\ga,\ga]$, where once again $P_{p+1}$ is the projection on weight $(p+1)$, and the bracket is the one on $\Pol^t(f,n)$. The same formula defines a section
$$\TCois^{\leq p} \to \Obs(\TCois,p+1)$$
of the obstruction bundle $\uObs(\TCois,p+1)$.

The proof of the following lemma is completely analogous to the proof of Lemma \ref{lm:lagrobs}.

\begin{lm}
There are natural equivalences between the spaces $\Cois^{\leq p+1}$, $\TCois^{\leq p+1}$ and the vanishing loci of the sections
$$\Cois^{\leq p} \to \uObs(\Cois,p+1)$$
and
$$\TCois^{\leq p} \to \uObs(\TCois,p+1)$$
defined above.
\end{lm}

Putting all of them together, we obtain a diagram
$$
\xymatrix{
\uObs(p+1, \TCois\times \Lagr) \ar[d] \ar[r] & \uObs(p+1, \Cois\times \Lagr) \ar[d] \\
\TCois(A, n)^{\leq p}\times \Lagr(A, n)^{\leq p} \ar@/_1pc/[u] \ar[r] & \Cois(A, n)^{\leq p}\times \Lagr(A, n)^{\leq p}.  \ar@/_1pc/[u] \ar@/_1.5pc/[l]_{\Phi_p} \\
}
$$

Moreover, the lemmas above show that taking the vanishing loci vertically we obtain the obvious projection
\[\TCois^{\leq p+1}\times \Lagr^{\leq p+1}\longrightarrow \Cois^{\leq p+1}\times \Lagr^{\leq p+1}.\]
On the other hand, the vanishing locus of the bottom section is by definition $\Comp^{\leq p}$. Our next goal is to show that the top map in the diagram also admits a canonical section.

Suppose we are given two elements $\ga$ and $\delta$ of $\Pol^t(f,n)$. By the results of \cite[Section 3.6]{MS1}, $\ga$ can be used to induce a commutative diagram of algebras
\[ \xymatrix{
\DR(A) \ar[r] \ar[d]^{\mu(-,\ga)_{A}} & \DR(B) \ar[d]^{\mu(-,\ga)_B} \\
\Pol^t(A,n) \ar[r] & \Pol^t(B/A,n-1).
} \]
Similarly, $\delta$ defines derivations $\psi_{\delta,A}$ and $\psi_{\delta,B}$ on $\Pol^t(A,n)$ and $\Pol^t(B/A,n-1)$, compatible with the $\bP_{[n+2,n+1]}$-structure. In particular, their restriction to weight zero define a commutative diagram
\[ \xymatrix{
\Omega^1_A \ar[r] \ar[d] & \Omega^1_B \ar[d] \\
\Pol^t(A,n) \ar[r] & \Pol^t(B/A,n-1).
} \]
Putting all together, we can now construct two derivations which fit vertically in the diagram
$$\xymatrix{
\DR(A) \ar[r] \ar[d]^{\nu(-,\ga,\delta)_A} & \DR(B) \ar[d]^{\nu(-,\ga,\delta)_{B}} \\
\Pol^t(A,n) \ar[r] & \Pol^t(B/A,n-1).
}$$
More explicitly, $\nu(-, \gamma, \delta)_A$ is an $A$-linear derivation $\DR(A)\rightarrow \Pol^t(A, n)$ relative to $\mu(-, \gamma)_A$ and similarly for $\nu(-, \gamma, \delta)_B$. On weight 1 generators they are given by
$$\nu(a \ddr x, \ga, \delta)_A = \mu(a,\ga)_A\psi_{\delta,A}(x) = a\psi_{\delta,A}(x) $$
and
$$\nu(b \ddr y, \ga, \delta)_B = \mu(b,\ga)_B\psi_{\delta,B}(y)=b\psi_{\delta,B}(y).$$

Passing to homotopy fibers, we obtain a map
$$\nu(-, \ga, \delta)\colon \DR(f) \to \Pol^t(f,n).$$ 
Define the morphism
$$\uObs(p+1, \Cois\times \Lagr)\to \uObs(p+1, \TCois\times \Lagr)$$
to be given by
$$(\ga, \la, \delta_\ga, \delta_\la)\mapsto (\ga + \e \si(\ga) - \e\mu(\la, \ga), \la, \delta_\ga + \e\si(\delta_\ga) - \e \nu(\lambda, \ga, \delta_\ga) - \e\mu(\delta_\la, \ga), \delta_\la),$$
where we denoted by $\delta_\ga$ and $\delta_\la$ the elements of the obstruction spaces.
Notice that the above formula clearly gives a section of the projection
$$ \uObs(p+1, \TCois\times\Lagr) \rightarrow \uObs(p+1, \Cois \times \Lagr).$$
In particular, we obtain a diagram
\

$$
\xymatrix{
\uObs(p+1, \TCois\times \Lagr) \ar[d] \ar[r] & \uObs(p+1, \Cois\times \Lagr) \ar@/_1.3pc/[l] \ar[d] \\
\TCois(A, n)^{\leq p}\times \Lagr(A, n)^{\leq p} \ar@/_1pc/[u] \ar[r] & \Cois(A, n)^{\leq p}\times \Lagr(A, n)^{\leq p} \ar@/_1.3pc/[l]_{\Phi_p} \ar@/_1pc/[u]
}
$$

We are now going to show that the top section is compatible with other sections. First, we need a preliminary Lemma.

\begin{lm}
Let $\ga \in \Pol^t(f,n)^{\leq p}$ and $\la \in \DR(f)^{\leq p}$, and suppose $\ga$ is of degree $n+2$. Then we have
$$\frac 12\nu(\la,\ga, [\ga,\ga]) + \mu(\ddr \la, \ga) = [\ga, \mu(\la, \ga)].$$
\label{lm:munu}
\end{lm}
\begin{proof}
For fixed $\ga$, the three terms are induced by pairs of maps $\DR(A) \to \Pol^t(A,n)$ and $\DR(B) \to \Pol^t(B/A,n-1)$, which are all derivations determined by their values in weight 0 and 1. It follows that it is enough to prove the lemma for $\la$ in weight 0 or 1.
\begin{itemize}
\item Suppose $\la \in \DR(f)$ is of weight 0.
By definition $\nu(-,\ga,[\ga,\ga])$ is induced by the pair of derivations $\psi_{[\ga,\ga],A}$ and $\psi_{[\ga,\ga],B}$, which are respectively $A$-linear and $B$-linear. In particular, they are both zero in weight 0, so that also $\nu(\la,\ga,[\ga,\ga])$ is zero if $\la$ is of weight 0.

We are thus left with proving that if $\la$ has weight 0, then
$$\mu(\ddr \lambda, \gamma) = [\gamma, \mu(\lambda, \gamma)] = [\gamma, \lambda]$$
which follows from the definition of $\mu$.

\item  Suppose now $\la \in \DR(f)$ is a generator of weight 1 of the form $\la=\ddr g$. 
Then $\mu(\ddr \la, \ga) = 0$, and we need to prove that
$$\frac{1}{2}\nu(\ddr g,\ga, [\ga,\ga]) = [\ga, \mu(\ddr g, \ga)].$$

Using the definition of $\nu$ and $\mu$ this reduces to
\[\frac{1}{2}[[\gamma, \gamma], g] = [\gamma, [\gamma, g]]\]
which follows from the Jacobi identity.
\end{itemize}
\end{proof}

\begin{prop}
The section $\uObs(p+1, \Cois\times \Lagr)\rightarrow \uObs(p+1, \TCois\times \Lagr)$ constructed above commutes with the obstruction maps. Moreover, the induced map on vanishing loci of the vertical sections is equivalent to
$$\Phi_{p+1}\colon \Cois(f, n)^{\leq (p+1)}\times \Lagr(f, n)^{\leq (p+1)}\longrightarrow \TCois(f, n)^{\leq (p+1)} \times \Lagr(f, n)^{\leq (p+1)}.$$
\end{prop}
\begin{proof}
We begin by analyzing the commutativity of
\[
\xymatrix{
\Cois^{\leq p}\times \Lagr^{\leq p} \ar[r] \ar^{\Phi_p}[d] & \uObs(p+1, \Cois \times \Lagr) \ar[d] \\
\TCois^{\leq p}\times \Lagr^{\leq p} \ar[r] & \uObs(p+1, \TCois\times \Lagr).
}
\]

Consider an element $(\ga, \la)\in \Cois^{\leq p}\times \Lagr^{\leq p}$. Its image in $\uObs(p+1, \Cois \times \Lagr)$ is
\[\left(\ga, \la, \frac 12 P_{p+1} [\ga, \ga], \ddr \lambda_p\right).\]
If we send it to $\uObs(p+1, \TCois \times \Lagr)$ via the top section, we get 
\begin{align*}
(&\ga + \e \si(\ga) - \e\mu(\la, \ga), \la, \\
&\frac{1}{2} P_{p+1}[\ga, \ga] + \e \frac p2 P_{p+1}[\ga, \ga] - \e \nu(\la, \ga, \frac 12 P_{p+1}[\ga, \ga]) - \e \mu(\ddr \la_p, \ga), \ddr \la_p).
\end{align*}

Alternatively, we can apply $\Phi_p$ to the pair $(\ga, \la)$, getting
$$\Phi_p(\ga, \la) = (\ga + \e \si(\ga) - \e \mu(\la, \ga), \la).$$
The image of $\Phi_p(\ga, \la)$ in $\uObs(p+1, \TCois \times \Lagr)$ is
$$(\ga + \e \si(\ga) - \e \mu(\la, \ga), \la, \frac 12 P_{p+1}[\ga, \ga] + \e P_{p+1}[\ga, \si(\ga) - \mu(\la, \ga)], \ddr \la_p).$$

This means that we are left with proving that
$$\frac p2 P_{p+1}[\ga, \ga] - \nu(\la, \ga, \frac 12 P_{p+1}[\ga, \ga]) - \mu(\ddr \la_p, \ga) = P_{p+1}[\ga, \si(\ga) - \mu(\la, \ga)]$$
in $\Pol^t(f, n)^{p+1}$. A straightforward computation shows that $[\ga, \si(\ga)] = \frac{1}{2}\si( [\ga, \ga])$, so that $\frac p2 [\ga, \ga] =P_{p+1}[\ga, \si(\ga)]$ in $\Pol^t(f, n)^{p+1}$. The rest of the terms are dealt with using Lemma \ref{lm:munu}.

This proves the first part of the proposition. For the second statement, let $(\ga', \lambda')$ be an element in $\Cois^{\leq p+1}\times \Lagr^{\leq p+1}$, and write
$$\ga' = \ga + \ga_{p+1} \ \ \text{and} \ \ \la'= \la +\la_{p+1},$$
where $\ga$ and $\la$ have no component of weight $p+1$, while $\ga_{p+1}$ and $\la_{p+1}$ are concentrated in weight $p+1$.
Then the image of $(\ga',\la')$ under $\Phi_{p+1}$ in $\TCois^{\leq(p+1)}\times \Lagr^{\leq(p+1)}$ is
$$(\ga+\ga_{p+1} + \epsilon\si(\ga + \ga_{p+1}) - \epsilon\mu(\la + \la_{p+1}, \ga + \ga_{p+1}), \la+\la_{p+1}).$$

By weight reasons
$$\mu(\la + \la_{p+1}, \ga + \ga_{p+1}) = \mu(\la, \ga) + \mu(\la_{p+1}, \ga) + \nu(\la, \ga, \ga_{p+1})$$
and hence the induced map on vanishing loci coincides with $\Phi_{p+1}$, which concludes the proof.
\end{proof}

Thanks to the proposition, we end up with a diagram
\newline
$$\xymatrix{
\uObs(p+1, \TCois \times \Lagr) \ar[d] \ar[r] & \uObs(p+1, \Cois \times \Lagr) \ar@/_1.3pc/[l] \ar[d] \\
\TCois^{\leq p}\times \Lagr^{\leq p} \ar@/_1pc/[u] \ar[r] & \Cois^{\leq p} \times \Lagr^{\leq p} \ar@/_1.3pc/[l]_{\Phi_{p}} \ar@/_1pc/[u] \\
\TCois^{\leq (p+1)} \times \Lagr^{\leq (p+1)} \ar[u] \ar[r] & \Cois^{\leq (p+1)}\times \Lagr^{\leq (p+1)} \ar[u] \ar@/_1.3pc/[l]_{\Phi_{p+1}}
}
$$

By definition, the vanishing loci of the horizontal sections $\Phi_p$ and $\Phi_{p+1}$ are $\Comp^{\leq p}$ and $\Comp^{\leq p+1}$ respectively. Let us denote by $\uObs(\Comp,p+1)$ the horizontal vanishing locus of the top section. Therefore, we obtain a diagram

\[
\xymatrix{
\uObs(p+1, \TCois \times \Lagr) \ar[d] \ar[r] & \uObs(p+1, \Cois \times \Lagr) \ar[d] \ar@/_1.3pc/[l] & \uObs(p+1, \Comp) \ar[d] \ar[l] \\
\TCois^{\leq p}\times \Lagr^{\leq p} \ar@/_1pc/[u] & \Cois^{\leq p} \times \Lagr^{\leq p} \ar@/_1.3pc/_{\Phi_p}[l] \ar@/_1pc/[u] & \Comp^{\leq p} \ar[l] \ar@/_1pc/[u] \\
\TCois^{\leq (p+1)} \times \Lagr^{\leq (p+1)} \ar[u] & \Cois^{\leq (p+1)}\times \Lagr^{\leq (p+1)} \ar[u] \ar@/_1.3pc/_{\Phi_{p+1}}[l] & \Comp^{\leq(p+1)} \ar[u] \ar[l]
}
\]

By the commutation of limits, we see that $\Comp^{\leq(p+1)}\rightarrow \Comp^{\leq p}$ realizes $\Comp^{\leq(p+1)}$ as a vanishing locus of
\[
\xymatrix{
\Comp^{\leq p} \ar@/^1pc/[r]& \uObs(\Comp,p+1).  \ar[l]
}
\]

By definition, the obstruction bundle $\uObs(\Comp,p+1)$ has a quite explicit description: its fiber over a compatible pair $(\ga, \la) \in \Comp^{\leq p}$ is given by the homotopy fiber of
$$\begin{array}{ccc}
\Pol^t(f, n)^{p+1}\oplus \DR(f)^{p+1} &\longrightarrow &\Pol^t(f, n)^{p+1} \\
(\ga_{p+1}, \la_{p+1}) &\longmapsto &\si(\ga_{p+1}) - \nu(\la_2, \ga_2, \ga_{p+1}) - \mu(\la_{p+1}, \ga_2).
\end{array}$$

Note that by weight reasons $\nu(\lambda, \gamma, \gamma_{p+1}) = \nu(\lambda_2, \gamma_2, \gamma_{p+1})$ and $\mu(\lambda_{p+1}, \gamma) = \mu(\lambda_{p+1}, \gamma_2)$ in $\Pol^t(f, n)^{p+1}$. Comparing the obstruction bundles for $\Comp$ and $\Lagr$, we get a diagram
$$
\xymatrix{
\uObs(p+1, \Comp) \ar[d] \ar[r] & \uObs(p+1, \Lagr) \ar[d] \\
\Comp^{\leq p} \ar@/_1pc/[u] \ar[r] & \Lagr^{\leq p} \ar@/_1pc/[u] \\
\Comp^{\leq (p+1)} \ar[u] \ar[r] & \Lagr^{\leq (p+1)} \ar[u]
}
$$
where the horizontal arrows are given by the natural projections.

We denote by $\uObs(p+1,\Comp^{nd})$ the restriction of the obstruction bundle $\uObs(p+1,\Comp)$ to the subspace $\Comp^{nd} \subset \Comp$.

\begin{prop}\label{prop:equivofobstr}
The projection
$$\uObs(p+1, \Comp^{nd})\rightarrow \uObs(p+1, \Lagr)$$
is an equivalence.
\end{prop}
\begin{proof}
Let $(\ga, \lambda)$ be a non-degenerate compatible pair in weight $\leq p$. The proposition is equivalent to showing that the natural projection
$$\Pol^t(f,n)^{p+1} \oplus \DR(f)^{p+1} \to \DR(f)^{p+1}$$
is an equivalence.
It is then enough to show that the map
$$
\begin{array}{ccc}
\Pol^t(f,n)^{p+1} & \to & \Pol^t(f,n)^{p+1} \\
x & \mapsto & \si(x) - \nu(\la_2, \ga_2, x)
\end{array}$$
is an equivalence.

We claim that $\nu(\lambda_2, \gamma_2, x)$ is equivalent to $(p+1)x$ if $x$ has weight $p+1$. To show this, it is enough to show the claim for its components $\nu(\lambda_2, \gamma_2, x)_A$ and $\nu(\lambda_2, \gamma_2, x)_B$.

By construction, if $\lambda$ has weight 1, $\nu(\lambda, \gamma_2, x)_A$ is a derivation in $x\in\Pol^t(A, n)$. But $\nu(\lambda, \gamma_2, x)_A$ is a derivation in $\lambda$ relative to $\mu(-, \gamma_2)_A$. Since $\DR(A)$ is generated in weight 1, we conclude that $\nu(\lambda, \gamma_2, x)_A$ is a derivation in $x$ for $\lambda$ of any weight. Therefore, it is enough to prove that $\nu(\lambda_2, \gamma_2, x)_A$ is homotopic to $px$ for $x$ of weight $p=0$ and $p=1$. The case $p=0$ is clear since then $\nu(\lambda_2, \gamma_2, x)_A = 0$. Now suppose $x$ has weight $1$. Let
\[\lambda_A = \sum_{i,j} a_{ij} \ddr f_i \ddr f_j\]
be the two-form on $A$ underlying $\lambda_2$ and $\gamma_A$ the bivector on $A$ underlying $\gamma_2$. Then
\[\nu(\lambda_2, \gamma_2, x)_A = \pm\sum_{i, j} 2a_{ij} [\gamma_A, f_i] [x, f_j].\]

But $\sum_{i, j} 2a_{ij} [\gamma_2, f_i] \ddr f_j\in T_A\otimes \Omega^1_A\cong \Hom(T_A, T_A)$ coincides with $\gamma_A^\sharp\circ\lambda_A^\sharp$. Since $\gamma_2$ and $\lambda_2$ are compatible, $\gamma_A^\sharp\circ \lambda_A^\sharp\cong \id$ and hence $\nu(\lambda_2, \gamma_2, x)_A\cong x$. The claim for $\nu(\lambda_2, \gamma_2, x)_B$ is proved similarly.

We obtain that
$$
\begin{array}{ccc}
\Pol^t(f,n)^{p+1} & \to & \Pol^t(f,n)^{p+1} \\
x & \mapsto & \si(x) - \nu(\la_2, \ga_2, x)
\end{array}$$
is homotopic to $(p - (p+1))\id = -\id$, and it is thus an equivalence.
\end{proof}

Finally, we obtain the required statement.

\begin{prop}
The projection
$$\Comp^{nd}(f,n) \to \Lagr(f,n)$$
is an equivalence.
\end{prop}
\begin{proof}
By Proposition \ref{prop:equivweight2} the map $\Comp^{nd}(f, n)\rightarrow \Lagr(f, n)^{\leq 2}$ is an equivalence. Suppose that we have already proved that
\[\Comp^{nd}(f, n)^{\leq p}\rightarrow \Lagr(f, n)^{\leq p}\]
is an equivalence for some $p$. We have a diagram
\[
\xymatrix{
\uObs(p+1, \Comp^{nd}) \ar[d] \ar^{\sim}[r] & \uObs(p+1, \Lagr) \ar[d] \\
\Comp^{nd, \leq p}(f, n) \ar@/_1pc/[u] \ar^{\sim}[r] & \Lagr^{\leq p}(f, n) \ar@/_1pc/[u] \\
\Comp^{nd, \leq (p+1)}(f, n) \ar[u] \ar[r] & \Lagr^{\leq (p+1)}(f, n) \ar[u]
}
\]
where the bottom row is obtained as a vanishing locus of the vertical maps and the top map is an equivalence by Proposition \ref{prop:equivofobstr}. Therefore, $\Comp^{nd, \leq (p+1)}(f, n)\rightarrow \Lagr^{\leq (p+1)}(f, n)$ is also an equivalence.

We have
\[\Comp^{nd}(f, n) = \lim_p \Comp^{nd, \leq p}(f, n),\qquad \Lagr(f, n) = \lim_p \Lagr^{\leq p}(f, n)\]
and therefore $\Comp^{nd}(f, n)\rightarrow \Lagr(f, n)$ is also an equivalence.
\end{proof}

These results, together with Proposition \ref{prop:comp=cois}, conclude the proof of Theorem \ref{thm:coisnd=lagr2}, which in turn finally implies Theorem \ref{thm:coisnd=lagr}.

\section{Quantization}

In this section we describe quantizations of $n$-shifted Poisson stacks and $n$-shifted coisotropic structures and show that any $n$-shifted coisotropic structure admits a formal quantization for $n > 1$.

\subsection{Beilinson--Drinfeld operads}
\label{sect:BDoperads}

Recall that one has a family of dg operads $\bP_n$ for $n\in\Z$ which, as we recall, are Hopf. Therefore, the $\infty$-category $\balg_{\bP_n}$ is endowed with a symmetric monoidal structure. Similarly, one has a family of dg Hopf operads $\bE_n$ for $n\geq 0$ which are defined to be
\[\bE_n = \C_\bullet(\E_n; k),\]
where $\E_n$ is the topological operad of little $n$-disks.

\begin{remark}
We consider unital versions of operads $\E_n$. For instance,
\[\bE_0(0) = k,\qquad \bE_0(1) = k,\qquad \bE_0(n) = 0,\ n>1.\]
Therefore, the operad $\bE_0$ controls complexes with a distinguished element.
\end{remark}

The Beilinson--Drinfeld operads $\BD_n$ are operads providing an interpolation between the operads $\bP_n$ and $\bE_n$. That is, they are graded Hopf dg operads over $k\llbracket\hbar\rrbracket$ with $\hbar$ of weight $1$ together with equivalences
\[\BD_n/\hbar\cong \bP_n,\qquad \BD_n[\hbar^{-1}]\cong \bE_n\llpar \hbar\rrpar.\]

The known definition of the operads $\BD_n$ is non-uniform in $n$ and they are defined separately for $\BD_0$, $\BD_1$ and $\BD_n$ for $n\geq 2$. The following is \cite[Definition 2.4.0.1]{CG}.
\begin{defn}
A \emph{$\BD_0$-algebra} is a dg $k\llbracket\hbar\rrbracket$-module together with a degree 1 Lie bracket $\{-,-\}$ and a unital commutative multiplication satisfying the relations
\begin{itemize}
\item $\d(ab) = \d(a) b + (-1)^{|a|} a \d(b) + \hbar\{a, b\}$,

\item $\{x, yz\} = \{x, y\} z + (-1)^{|y||z|}\{x, z\} y$.
\end{itemize}
\end{defn}

These relations define a dg operad $\BD_0$. We introduce a weight grading by assigning weight $-1$ to $\{-,-\}$ and weight $0$ to the multiplication. Note that we have an isomorphism $\BD_0\cong \bP_0\llbracket\hbar\rrbracket$ of bigraded operads; we transfer the Hopf structure from $\bP_0$ to one on $\BD_0$ using this isomorphism.

The following definition is given in \cite[Section 2.4.2]{CG}.
\begin{defn}
A $\BD_1$-algebra is a dg Lie algebra $(A, \{-,-\})$ over $k\llbracket\hbar\rrbracket$ equipped with an associative multiplication satisfying the relations
\begin{itemize}
\item $\hbar\{x, y\} = xy - (-1)^{|x||y|} yx$,

\item $\{x, yz\} = \{x, y\} z + (-1)^{|y||z|} \{x, z\} y$.
\end{itemize}
\end{defn}

We introduce an additional weight grading such that $\{-, -\}$ has weight $-1$ and that the multiplication has weight $0$, thus making $\BD_1$ into a graded dg operad. The Hopf structure on $\BD_1$ is defined so that
\[\Delta(m) = m\otimes m,\qquad \Delta(\{-,-\}) = \{-,-\}\otimes m + m\otimes \{-,-\},\]
where $m\in\BD_1(2)$ is the product and $\{-,-\}\in\BD_1(2)$ is the Lie bracket.

Finally, $\E_n$ has a Postnikov tower which gives a Hopf filtration of $\bE_n$ and we define $\BD_n$ for $n\geq 2$ to be the graded operad obtained as the Rees construction with respect to this filtration.

\begin{remark}
Note that with respect to this filtration $\gr \bE_1\cong \bE_1$ while $\bP_1\not\cong \bE_1$, i.e. the filtration on $\bE_1$ induced by the operad $\BD_1$ is different from the Postnikov filtration. The same remark applies to the case $n=0$.
\end{remark}

Let us now state two important results about the operads $\BD_n$.

\begin{thm}[Formality of the operad of little $n$-disks]
Suppose $n\geq 2$. Then one has an equivalence of graded Hopf dg operads
\[\BD_n\cong \bP_n\llbracket\hbar\rrbracket\]
compatible with the equivalence $\BD_n/\hbar\cong \bP_n$.
\label{thm:Enformality}
\end{thm}

The statement for $n=2$ was proved by Tamarkin \cite{Ta} using the existence of rational Drinfeld associators, by Kontsevich \cite{Ko} and Lambrechts--Voli\'{c} \cite{LV} for all $n\geq 2$ and $k\supset \mathbb{R}$ and finally by Fresse--Willwacher \cite{FW} for all $n\geq 2$ and all fields $k$ of characteristic zero.

\begin{remark}
The space of formality isomorphisms $\bE_n\cong \bP_n$ as dg Hopf operads is nontrivial and is described in \cite[Corollary 5]{FTW}. For instance, the space of formality isomorphisms $\bE_2\cong \bP_2$ has connected components parametrized by the set of Drinfeld associators.
\end{remark}

The following result has been announced by Rozenblyum:
\begin{thm}
Suppose $n\geq 0$. Then one has an equivalence of $k\llbracket\hbar\rrbracket$-linear symmetric monoidal $\infty$-categories
\[\balg_{\BD_{n+1}}(\cM)\cong \balg(\balg_{\BD_n}(\cM)).\]
\label{thm:BDadditivty}
\end{thm}

Inverting $\hbar$ in this equivalence, we obtain an equivalence of symmetric monoidal $\infty$-categories
\[\balg_{\bE_{n+1}}(\cM)\cong \balg(\balg_{\bE_n}(\cM))\]
constructed by Dunn and Lurie (see \cite[Theorem 5.1.2.2]{Lu}). In the other extreme, setting $\hbar=0$ we obtain an equivalence of symmetric monoidal $\infty$-categories
\[\balg_{\bP_{n+1}}(\cM)\cong \balg(\balg_{\bP_n}(\cM))\]
constructed in \cite{Sa2}.

\begin{remark}
It is expected that for $n\geq 2$ the equivalence of Theorem \ref{thm:BDadditivty} is compatible with Theorem \ref{thm:Enformality} in that the following diagram commutes:
\[
\xymatrix{
\balg_{\BD_{n+1}}(\cM)\ar^{\sim}[r] \ar^{\sim}[d] & \balg_{\bP_{n+1}}(\cM\llbracket\hbar\rrbracket) \ar^{\sim}[d] \\
\balg(\balg_{\BD_n}(\cM)) \ar^{\sim}[r] & \balg(\balg_{\bP_n}(\cM\llbracket\hbar\rrbracket))
}
\]
where the vertical functor on the right is given by Poisson additivity, i.e. the statement of Theorem \ref{thm:BDadditivty} for $\hbar=0$.
\end{remark}

\subsection{Deformation quantization for Poisson structures}

Once we have the definition of the $\BD_n$ operads we can define the notion of deformation quantization.

\begin{defn}
Let $A$ be a $\bP_{n+1}$-algebra in $\cM$. A \emph{deformation quantization} of $A$ is a $\BD_{n+1}$-algebra $A_\hbar$ together with an equivalence of $\bP_{n+1}$-algebras $A_\hbar/\hbar\cong A$.
\end{defn}

\begin{remark}
Suppose $n=0$ and fix a (non-dg) Poisson algebra $A$. A deformation quantization of $A$ in the classical sense is given by an associative algebra $A_\hbar$ flat over $k\llbracket\hbar\rrbracket$ whose multiplication is commutative at $\hbar=0$ together with an isomorphism $A_\hbar/\hbar\cong A$ of Poisson algebras, where $A_\hbar/\hbar$ is equipped with the induced Lie bracket $(ab - ba)/\hbar$. It is easy to see that one can therefore lift $A_\hbar$ to a $\BD_1$-algebra and the flatness condition is necessary to ensure that the derived tensor product $A_\hbar\otimes_{k\llbracket\hbar\rrbracket} k$ coincides with the ordinary tensor product.
\end{remark}

One can similarly give a definition for general stacks, see \cite[Section 3.5.1]{CPTVV}.
\begin{defn}
Let $X$ be a derived Artin stack equipped with an $n$-shifted Poisson structure. A \emph{deformation quantization} of $X$ is a lift of the $\bP_{n+1}$-algebra $\cB_X(\infty)$ in the $\infty$-category of $\bD_{X_{DR}}(\infty)$-modules to a $\BD_{n+1}$-algebra $\cB_{X, \hbar}(\infty)$.
\end{defn}

The following is \cite[Theorem 3.5.4]{CPTVV} and immediately follows from Theorem \ref{thm:Enformality}.
\begin{thm}
Let $X$ be a derived Artin stack equipped with an $n$-shifted Poisson structure for $n > 0$. Then deformation quantizations exist.
\end{thm}

Given a derived Artin stack $X$ we have a symmetric monoidal $\infty$-category $\Perf(X)$ of perfect complexes which by \cite[Corollary 2.4.12]{CPTVV} can be identified with
\[\Perf(X)\cong \Gamma(X_{DR}, \bmod_{\cB_X(\infty)}^{perf}),\]
where $\bmod_{\cB_X(\infty)}^{perf}$ is the prestack on $X_{DR}$ of symmetric monoidal $\infty$-categories of perfect $\cB_X(\infty)$-modules.

Now suppose $X$ has an $n$-shifted Poisson structure which admits a deformation quantization. Then $\cB_X(\infty)$ has a deformation over $k\llbracket\hbar\rrbracket$ as an $\bE_{n+1}$-algebra. Therefore, using the equivalence $\balg_{\E_{n+1}}\cong \balg_{\E_n}(\balg)$ we obtain that $\bmod_{\cB_X(\infty)}^{perf}$ has a deformation $\bmod_{\cB_{X, \hbar}(\infty)}$ as a prestack of $\E_n$-monoidal $\infty$-categories. Hence we conclude that $\Perf(X)$ itself inherits a deformation over $k\llbracket\hbar\rrbracket$ as an $\E_n$-monoidal $\infty$-category.

\subsection{Deformation quantization for coisotropic structures}

Now we define deformation quantization in the relative setting. Recall that by \cite[Theorem 3.7]{Sa2} one can identify $\bP_{[n+1, n]}$-algebras with a pair of an associative algebra $A$ and a left $A$-module $B$ in the $\infty$-category of $\bP_n$-algebras.

\begin{defn}
Let $(A, B)$ be a $\bP_{[n+1, n]}$-algebra. A \emph{deformation quantization} of $(A, B)$ is a pair $(A_\hbar, B_\hbar)\in\blmod(\balg_{\BD_n})$ together with an equivalence $(A_\hbar, B_\hbar)/\hbar\cong (A, B)$ of objects of $\blmod(\balg_{\bP_n})$.
\end{defn}

Note that a deformation quantization of a $\bP_{[n+1, n]}$-algebra $(A, B)$ in particular gives a deformation quantization of the $\bP_{n+1}$-algebra $A$ by Theorem \ref{thm:BDadditivty}.

\begin{defn}
Let $f\colon L\rightarrow X$ be a morphism of derived Artin stacks equipped with an $n$-shifted coisotropic structure. A \emph{deformation quantization} of $f$ is given by a deformation quantization $\cB_{X,\hbar}(\infty)$ of $X$  and a deformation quantization $(f^*\cB_{X, \hbar}(\infty), \cB_{L, \hbar}(\infty))$ of the $\bP_{[n+1, n]}$-algebra $(f^*\cB_X(\infty), \cB_L(\infty))$ in the $\infty$-category of $\bD_{L_{DR}}(\infty)$-modules.
\end{defn}

As in the case of shifted Poisson structures, we have the following obvious result:
\begin{thm}
Let $f\colon L\rightarrow X$ be a morphism of derived Artin stacks equipped with an $n$-shifted coisotropic structure for $n>1$. Then deformation quantizations of $f$ exist.
\label{thm:coisotropicdefquantization}
\end{thm}
\begin{proof}
Indeed, Theorem \ref{thm:Enformality} gives an equivalence of symmetric monoidal $\infty$-categories
\[\blmod(\balg_{\bP_n}(\cM\llbracket\hbar\rrbracket))\cong \blmod(\balg_{\BD_n}(\cM))\]
and we have a functor $\balg_{\bP_n}(\cM)\rightarrow \balg_{\bP_n}(\cM\llbracket\hbar\rrbracket)$ sending a $\bP_n$-algebra $A$ to $A\otimes k\llbracket\hbar\rrbracket$.

Combining these two functors, we obtain a functor
\[\blmod(\balg_{\bP_n}(\cM))\longrightarrow \blmod(\balg_{\BD_n}(\cM))\]
which gives a deformation quantization of any $\bP_{[n+1, n]}$-algebra. The result follows by applying the functor to the $\bP_{[n+1, n]}$-algebra $(f^*\cB_X(\infty), \cB_L(\infty))$ and noticing that the $\BD_{n+1}$-algebra obtained as the quantization of $f^*\cB_X(\infty)$ is canonically equivalent to the pullback under $f$ of the quantization of $\cB_X(\infty)$.
\end{proof}

Given a morphism of derived Artin stacks $f\colon L\rightarrow X$, the symmetric monoidal functor $f^*\colon \Perf(X)\rightarrow \Perf(L)$ can be realized as the composite
\begin{align*}
\Perf(X)&\cong \Gamma(X_{DR}, \bmod_{\cB_X(\infty)}^{perf}) \\
&\rightarrow \Gamma(L_{DR}, \bmod_{f^*\cB_X(\infty)}^{perf}) \\
&\rightarrow \Gamma(L_{DR}, \bmod_{\cB_L(\infty)}^{perf}) \\
&\cong \Perf(L).
\end{align*}

Note that $f^*\colon\Perf(X)\rightarrow \Perf(L)$ promotes the pair $(\Perf(X), \Perf(L))$ to an object of $\blmod(\balg_{\E_{n-1}}(\St))$, where $\St$ is the $\infty$-category of small stable dg categories. Therefore, as before we see that given an $n$-shifted coisotropic structure we get a deformation over $k\llbracket\hbar\rrbracket$ of the pair $(\Perf(X), \Perf(L))$ as an object of $\blmod(\balg_{\E_{n-1}}(\St))$. A bit more explicit description of this $\infty$-category is given by the following statement.

\begin{conjecture}
Let $\SC_n$ be the $n$-dimensional topological Swiss--cheese operad. Then one has equivalences of symmetric monoidal $\infty$-categories
\[\balg_{\SC_{n+m}}\cong \balg_{\E_n}(\balg_{\SC_m})\cong \balg_{\SC_m}(\balg_{\E_n}).\]
\end{conjecture}

This conjecture is a slight generalization of the Dunn--Lurie additivity statement to stratified factorization algebras. Given this statement we see that a quantization of an $n$-shifted coisotropic structure gives rise to a deformation of the pair $(\Perf(X), \Perf(L))$ as an algebra over $\SC_n$.

\begin{example}
Let $G\subset D$ be an inclusion of affine algebraic groups such that $\fd=\Lie(D)$ carries an element $c\in\Sym^2(\fd)^D$ and the inclusion $\g=\Lie(G)\subset \fd$ is coisotropic with respect to $c$. Then it is shown in \cite[Proposition 2.9]{Sa3} that $\B D$ carries a 2-shifted Poisson structure and $\B G\rightarrow \B D$ a coisotropic structure. In particular, we may apply Theorem \ref{thm:coisotropicdefquantization} to this coisotropic.

Fix a Drinfeld associator which provides an equivalence of symmetric monoidal $\infty$-categories
\[\balg_{\bE_2}\cong \balg_{\bP_2}.\]
Therefore, we obtain an equivalence of $\infty$-categories
\[\blmod(\balg_{\bE_2})\cong \blmod(\balg_{\bP_2}).\]
Thus, we obtain a quantization $\widetilde{\Rep}(D)$ of $\B D$ which is a braided monoidal category and a quantization $\widetilde{\Rep}(G)$ of $\B G$ which is a monoidal category. Moreover, $\widetilde{\Rep}(G)$ becomes a monoidal module category over $\widetilde{\Rep}(D)$.
\end{example}


\begin{thebibliography}{CPTVV}

\bibitem[BG]{BG} V. Baranovsky, V. Ginzburg, Gerstenhaber-Batalin-Vilkoviski structures on coisotropic intersections, Math. Res. Lett. \textbf{17} (2010) 211--229, \href{http://arxiv.org/pdf/0907.0037.pdf}{arXiv:0907.0037}.

\bibitem[BGKP]{BGKP} V. Baranovsky, V. Ginzburg, D. Kaledin, J. Pecharich, Quantization of line bundles on Lagrangian subvarieties, Selecta Math. \textbf{22} (2016) 1--25, \href{https://arxiv.org/abs/1403.3493}{arXiv:1403.3493}.

\bibitem[CF]{CF} A. Cattaneo, G. Felder, Relative formality theorem and quantisation of coisotropic submanifolds, Adv. in Math. \textbf{208} (2007) 521--548, \href{http://arxiv.org/abs/math/0501540}{arXiv:math/0501540}.

\bibitem[CG]{CG} K. Costello, O. Gwilliam, \emph{Factorization algebras in quantum field theory}, vol. 2 (2016). Available at \url{http://people.mpim-bonn.mpg.de/gwilliam/vol2may8.pdf}.

\bibitem[CPTVV]{CPTVV} D. Calaque, T. Pantev, B. To\"en, M. Vaqui\'e, G. Vezzosi, Shifted Poisson structures and deformation quantization, J. Topol. \textbf{10} (2017) 483--584, \href{http://arxiv.org/abs/1506.03699}{arXiv:1506.03699}.

\bibitem[FTW]{FTW} B. Fresse, V. Turchin, T. Willwacher, The rational homotopy of mapping spaces of $E_n$ operads, \href{https://arxiv.org/abs/1703.06123}{arXiv:1703.06123}.

\bibitem[FW]{FW} B. Fresse, T. Willwacher, The intrinsic formality of $E_n$-operads, \href{https://arxiv.org/abs/1503.08699}{arXiv:1503.08699}.

\bibitem[Ha]{Ha} R. Haugseng, The higher Morita category of $E_n$-algebras, Geom. Topol. \textbf{21} (2017) 1631--1730, \href{https://arxiv.org/abs/1412.8459}{arXiv:1412.8459}.

\bibitem[JS]{JS} D. Joyce, P. Safronov, A Lagrangian Neighbourhood Theorem for shifted symplectic derived schemes, to appear in Annales Math. de la Facult\'e des Sciences de Toulouse, \href{http://arxiv.org/abs/1506.04024}{arXiv:1506.04024}.

\bibitem[Ko]{Ko} M. Kontsevich, Operads and motives in deformation quantization, Lett. Math. Phys. \textbf{48} (1999), 35--72, \href{https://arxiv.org/abs/math/9904055}{arXiv:math/9904055}.

\bibitem[Li]{Li} M. Livernet, Non-formality of the Swiss-Cheese operad, J. Topology \textbf{8} (2015) 1156--1166, \href{https://arxiv.org/abs/1404.2484}{arXiv:1404.2484}.

\bibitem[LV]{LV} P. Lambrechts, I. Voli\'{c}, Formality of the little $N$-disks operad, Mem. Amer. Math. Soc. \textbf{230} (2014) viii+116, \href{https://arxiv.org/abs/0808.0457}{arXiv:0808.0457}.

\bibitem[Lu]{Lu} J. Lurie, \emph{Higher Algebra}, available at \url{http://math.harvard.edu/~lurie/papers/HA.pdf}.

\bibitem[Me]{Me} V. Melani, Poisson bivectors and Poisson brackets on affine derived stacks, Adv. in Math. \textbf{288} (2016) 1097--1120, \href{http://arxiv.org/abs/1409.1863}{arXiv:1409.1863}.

\bibitem[MS1]{MS1} V. Melani, P. Safronov, Derived coisotropic structures I: affine case, to appear in Selecta Mathematica, \href{https://arxiv.org/abs/1608.01482}{arXiv:1608.01482}.

\bibitem[OP]{OP} Y.-G. Oh, J.-S. Park, Deformations of coisotropic submanifolds and strong homotopy Lie algebroids, Invent. Math. \textbf{161} (2005) 287--306, \href{http://arxiv.org/abs/math/0305292}{arXiv:math/0305292}.

\bibitem[Pri1]{Pri1} J. Pridham, Shifted Poisson and symplectic structures on derived $N$-stacks, J. Topol. \textbf{10} (2017) 178--210, \href{http://arxiv.org/abs/1504.01940}{arXiv:1504.01940}.

\bibitem[Pri2]{Pri2} J. Pridham, Deformation quantisation for $(-1)$-shifted symplectic structures and vanishing cycles, \href{https://arxiv.org/abs/1508.07936}{arXiv:1508.07936}.

\bibitem[Pri3]{Pri3} J. Pridham, Deformation quantisation for unshifted symplectic structures on derived Artin stacks, \href{https://arxiv.org/abs/1604.04458}{arXiv:1604.04458}.

\bibitem[Pri4]{Pri4} J. Pridham, Quantisation of derived Lagrangians, \href{http://arxiv.org/abs/1607.01000}{arXiv:1607.01000}.

\bibitem[PTVV]{PTVV} T. Pantev, B. To\"en, M. Vaqui\'e, G. Vezzosi, Shifted symplectic structures, Publ. Math. IHES \textbf{117} (2013) 271--328, \href{http://arxiv.org/abs/1111.3209}{arXiv:1111.3209}.

\bibitem[Sa1]{Sa1} P. Safronov, Poisson reduction as a coisotropic intersection, Higher Structures \textbf{1} (2017) 87--121, \href{http://arxiv.org/abs/1509.08081}{arXiv:1509.08081}.

\bibitem[Sa2]{Sa2} P. Safronov, Braces and Poisson additivity, \href{https://arxiv.org/abs/1611.09668}{arXiv:1611.09668}.

\bibitem[Sa3]{Sa3} P. Safronov, Poisson-Lie structures as shifted Poisson structures, \href{https://arxiv.org/abs/1706.02623}{arXiv:1706.02623}.

\bibitem[Sch]{Sch} C. Scheimbauer, Factorization homology as a fully extended topological field theory, Ph.D. thesis, available at \url{https://people.maths.ox.ac.uk/scheimbauer/ScheimbauerThesis.pdf}.

\bibitem[Ta]{Ta} D. Tamarkin, Formality of Chain Operad of Little Discs, Lett. Math. Phys. \textbf{66} (2003) 65--72, \href{https://arxiv.org/abs/math/9809164}{arXiv:math/9809164}.

\bibitem[TV]{TV} B. To\"en, G. Vezzosi, Homotopical algebraic geometry I: topos theory, Adv. in Math. \textbf{193} (2005), 257-372, \href{https://arxiv.org/abs/math/0207028}{arXiv:0207028}.

\bibitem[We]{We} A. Weinstein, Coisotropic calculus and Poisson groupoids, J. Math. Soc. Japan \textbf{40} (1988), 705--727.
\end{thebibliography}
\end{document}